\documentclass{amsart}
\usepackage[dvipsnames]{xcolor}
\usepackage{hyperref}
\hypersetup{
     colorlinks=true,
     linkcolor=blue!60!black,
     filecolor=blue!75!black,
     citecolor = green!50!black,      
     urlcolor=magenta,
     }

\usepackage{amssymb}
\usepackage{ bbold }
\usepackage[all]{xy}
\xyoption{all}
\usepackage{url}
\usepackage{mathrsfs}
\usepackage{tikz-cd}
\usepackage{eucal}
\usepackage{comment}
\usepackage{mathtools}

\counterwithin{equation}{section}
\title[The Balmer spectrum of integral permutation modules]{The Balmer spectrum of integral permutation modules}

\author[U. V Dubey]{Umesh V Dubey}
\address{
Harish-Chandra Research Institute, A CI of Homi Bhabha National Institute, Chhatnag Road, Jhunsi, Prayagraj 211019, India}
\address{
Homi Bhabha National Institute, Training School Complex, Anushakti Nagar, Mumbai 400 094, India}
\email{umeshdubey@hri.res.in}

\author[J. O. G\'omez]{Juan Omar G\'omez}
\thanks{}
\address{Fakult\"at f\"ur Mathematik,
Universit\"at Bielefeld, D-33501 Bielefeld, Germany.}
\email{jgomez@math.uni-bielefeld.de}

\newcommand{\comments}[1]{}

\newcommand{\colim}{\operatorname{colim}\nolimits}
\newcommand{\End}{\operatorname{End}\nolimits}

\newcommand{\Hom}{\operatorname{Hom}\nolimits}

\newcommand{\Ind}{\operatorname{Ind}\nolimits}
\newcommand{\Infl}{\operatorname{Infl}\nolimits}

\newcommand{\kos}{\operatorname{kos}\nolimits}

\newcommand{\Mod}{\operatorname{Mod}\nolimits}
\newcommand{\perm}{\operatorname{perm}\nolimits}
\newcommand{\Perm}{\operatorname{Perm}\nolimits}

\newcommand{\Res}{\operatorname{Res}\nolimits}
\newcommand{\Spc}{\operatorname{Spc}\nolimits}
\newcommand{\Spec}{\operatorname{Spec}\nolimits}

\newcommand{\restr}{\mathord\downarrow}

\def \A{{\mathcal A}}
\def \aa{\mathfrak a }

\def \bb{\mathfrak b }
\def \pp{\mathbf p }

\def \K{{\mathcal K}}
\def \N{{\mathbb N}} 
\def \mm{\mathfrak m }

\def \P{{\mathcal P}}
\def \pp{\mathfrak p }

\def \qq{\mathfrak q }

\def \T{{\mathcal T}}

\makeatletter
\newcommand*{\doublerightarrow}[2]{\mathrel{
  \settowidth{\@tempdima}{$\scriptstyle#1$}
  \settowidth{\@tempdimb}{$\scriptstyle#2$}
  \ifdim\@tempdimb>\@tempdima \@tempdima=\@tempdimb\fi
  \mathop{\vcenter{
    \offinterlineskip\ialign{\hbox to\dimexpr\@tempdima+1em{##}\cr
    \rightarrowfill\cr\noalign{\kern.5ex}
    \rightarrowfill\cr}}}\limits^{\!#1}_{\!#2}}}
\newcommand*{\triplerightarrow}[1]{\mathrel{
  \settowidth{\@tempdima}{$\scriptstyle#1$}
  \mathop{\vcenter{
    \offinterlineskip\ialign{\hbox to\dimexpr\@tempdima+1em{##}\cr
    \rightarrowfill\cr\noalign{\kern.5ex}
    \rightarrowfill\cr\noalign{\kern.5ex}
    \rightarrowfill\cr}}}\limits^{\!#1}}}
\makeatother

\newcommand{\nc}{\newcommand}
\nc{\dmo}{\DeclareMathOperator}
\nc{\Weyl}[2]{{#1}/\!\!/{#2}}
\nc{\WGH}{\Weyl{G}{H}}

\nc{\WGK}{\Weyl{G}{K}}
\nc{\WGL}{\Weyl{G}{L}}
\nc{\WGN}{\Weyl{G}{N}}

\theoremstyle{plain}

\newtheorem{theorem}{Theorem}[section]
\newtheorem{proposition}[theorem]{Proposition}
\newtheorem{corollary}[theorem]{Corollary}
\newtheorem{lemma}[theorem]{Lemma}

\theoremstyle{definition}
\newtheorem{definition}[theorem]{Definition}
\newtheorem{remark}[theorem]{Remark}
\newtheorem{recollection}[theorem]{Recollection}
\newtheorem{example}[theorem]{Example}

\newtheorem{notation}[theorem]{Notation}
\newtheorem{construction}[theorem]{Construction}
\newtheorem{convention}[theorem]{Convention}

\thanks{}

\begin{document}

\begin{abstract} 
We study the Balmer spectrum of the small derived category of permutation modules of a finite group over a commutative Noetherian ring. In this generality, we obtain a set-theoretic description of the spectrum. Under certain mild additional assumptions on the base ring, we describe its topology by reducing the problem first to the modular locus and then to the case of elementary abelian groups. Furthermore, we introduce a twisted cohomology ring in this context, which allows us to realize the spectrum, in the case of elementary abelian 
$p$-groups, as a Dirac scheme. This extends recent work of Balmer and Gallauer from the setting of ground fields to the more general setting of commutative Noetherian rings. 
\end{abstract}

\maketitle

\tableofcontents

\section{Introduction}

Let $G$ be a finite group and $R$ a commutative Noetherian ring. In a series of papers, Balmer and Gallauer have investigated various aspects of the so-called \emph{derived category of permutation $RG$--modules}, culminating in a description of its tensor-triangular (tt) geometry when the ground ring is a field. Their work also encompasses permutation modules for profinite groups; see \cite{BG23b,BG25}. In this paper, our main goal is to extend their analysis to finite groups over more general base rings.

The derived category of permutation $RG$--modules, denoted $\T(G;R)$, is a rigidly-compactly generated tensor triangulated category. Its subcategory of compact objects, $\K(G;R)$, coincides with the idempotent completion of the bounded homotopy category of finitely generated permutation $RG$--modules.

Why should one be interested in $\T(G;R)$? A compelling reason is its deep connections across various areas of equivariant mathematics. In fact, $\T(G;R)$ is one of several equivalent manifestations of the same structure. More precisely, $\T(G;R)$ is equivalent to:
\begin{itemize}
    \item[-] the derived category of cohomological $R$-linear Mackey functors for $G$;

  \item[-] the homotopy category of equivariant modules over the constant Green functor $H\underline{R}_G$; and

  \item[-] the derived category of $R$--linear Artin motives generated by the motives of intermediate fields in any Galois extension with Galois group $G$.
\end{itemize}
Among these, $\T(G;R)$ is arguably the most accessible. Permutation $RG$--modules are simply $R$--linearizations of finite $G$--sets, making the category $\T(G;R)$ relatively easy to define. However, this apparent simplicity is deceptive: the tt-geometry of $\K(G;R)$ remains difficult to describe, even when $R$ is a field, as detailed in \cite{BG23b}. This complexity may be anticipated, given the richness of the equivalent formulations listed above. Nevertheless, there is a concrete representation-theoretic reason underlying this difficulty. In \cite{BG22}, the authors show that wild categories, such as the bounded derived category of finitely generated $RG$--lattices and the stable module category, arise as localizations of $\K(G;R)$. This fact already reflects the intricate nature of $\K(G;R)$ within the context of group representation theory.

As previously mentioned, our goal is to understand the tt-geometry of permutation modules over an arbitrary commutative Noetherian ring. This paper should be regarded as a companion to the second-named author’s work \cite{Gom25}, which investigates the tt-geometry of the category $\T(G;R)$. That work focuses particularly on stratification and develops several tools that we employ here in the analysis of $\K(G;R)$. Accordingly, the present paper is devoted to the study of the Balmer spectrum of $\K(G;R)$.

We now outline the content of the paper in more detail.

\subsection*{The spectrum as a set} 
 Let $k$ be a field of characteristic $p$ dividing the order of the group $G$. For any $p$-subgroup $H\leq G$, the \textit{modular $H$-fixed points} tt-functor 
\[
\Psi^H_\pp\colon \T(G;k(\pp)) \to \T(\WGH;k(\pp)) 
\] 
 is constructed in \cite[Section 5]{BG23b}. Here $\WGH$ denotes $N_G(H)/H$. On the other hand, for any finite group $G$, there is a localization functor   
\[
\Upsilon^G\colon \K(G;k) \to \mathbf{D}_b(kG)
\]
see \cite{BG23a}. Balmer and Gallauer used these functors to describe the points of $\Spc(\K(G;k))$. In particular, \cite[Corollary 7.2]{BG23b} gives us that any point has the form 
$\P_G(H,\aa)\coloneqq \check\psi^H(\aa)$ for some $p$-subgroup $H\leq G$ and some prime $\aa \in \Spec^h(H^\bullet(\WGH;k))$, where $\check\psi^H$ denotes $\Spc( \Upsilon^{\WGH}\circ \Psi^H)$.

In \cite{Gom25}, the author showed that the points of $\Spc(\K(G;R))$ can be obtained from those of $\Spc(\K(G;k(\pp)))$ via the base change
geometric functors $\iota_\pp^\ast\colon \T(G;R)\to \T(G;k(\pp))$ over  all primes of $R$. In particular, one gets that a point in $\Spc(\K{G;R})$ has the form 
\[
\P_{G;R}(H,\aa,\pp)\coloneqq   \Spc(\iota_\pp^\ast)(\psi^H(\aa))
\]
for some $p$-subgroup $H\leq G$, and some $\aa\in\Spec^h(H^\bullet(\WGH;k(\pp)))$ provided that $\mathrm{char}(k(\pp))$ divides the order of the group; or $\P_{G;R}(1,\aa,\pp)=\Spc(\iota_\pp^\ast)(\aa)$, where $\aa$ is the unique point of $\Spc(\K(G;k(\pp)))$ provided that $\mathrm{char}(k(\pp))$ does not divide the order of the group.

We note that standard techniques from tensor-triangular geometry allow us to reduce our analysis to the case where $G$ is a $p$-group; see \cite[Lemma 6.1]{Gom25}. Accordingly, unless stated otherwise, we will henceforth restrict our attention to $p$-groups.

We are now prepared to state one of our main results: a set-theoretic description of the spectrum of $\K(G,R)$. It corresponds to Theorem \ref{primes in spc up to conj} and Proposition \ref{Spc as a set}.

\begin{theorem}
Let $G$ be a finite $p$-group  and $R$ be a commutative Noetherian ring. Then there is an explicit relation between  the points $\P_{G;R}(H,\aa,\pp)$
in $\Spc(\K(G;R))$ that only depends on $H$, $\aa$ and $\pp$.
In particular, we obtain that there is a set bijection induced by base change along all residue fields of $R$: 
   \[
   \coprod_{\pp\in\mathrm{Spec}(R)} \mathrm{Spc}(\K(G;k(\pp))) \xrightarrow[]{} \mathrm{Spc}(\K(G;R)).
   \]
\end{theorem}

The preceding theorem relies on the joint conservativity of the family $\{\iota_\pp^\ast\}_{\pp}$, and some of its consequences that are recorded in Theorem \ref{thm-weakly descent}. It also relies on the construction of the so-called \emph{Koszul objects}, which play a role analogous to that of Koszul objects in the work of Balmer and Gallauer. Remarkably, their original construction remains valid over an arbitrary commutative Noetherian base ring, provided that $p > 2$. In the case $p = 2$, a refinement is necessary; here, the construction is based on a finite permutation resolution of $R$, as developed in \cite{BG22}. This construction is detailed in Section~\ref{sec: Kosz}.

\subsection*{Topology of the spectrum} Since the space $\Spc(\K(G;R))$ is Noetherian, its topology is completely determined by the inclusion relations among prime ideals. In particular, the preceding result also provides some topological information. However, this is generally insufficient for a complete understanding of the topology. This challenge is already present in the work of Balmer and Gallauer, and addressing it requires the use of more sophisticated tools.

Our general approach is to understand how the topology on \textit{fibers} assembles into the global structure. More precisely, we investigate the topology of $\Spc(\K(G;R))$ via the family of maps induced by the base-change functors 
\[
\iota_\pp^\ast\colon \K(G;R)\to \K(G;k(\pp))
\]
and leverage results of Balmer and Gallauer in the case of fields to gain insight. To proceed, we first introduce some notation.

\begin{definition}
Let $G$ be a $p$-group, and $\eta\colon \mathbb Z \to R$ be the structural map. The \textit{modular locus} of $\K(G;R)$ is the (closed) subset 
   \[
   \Spc(\K(G;R))_p\coloneqq (\Spec(\eta)\circ \mathrm{Comp})^{-1}(p)\subseteq \Spc(\K(G;R))
   \]
where $\mathrm{Comp}$ denotes Balmer’s comparison map between the triangular spectrum and the Zariski spectrum. The  \textit{ordinary locus} of $\K(G;R)$ is the set-theoretic  complement of the modular locus. 
\end{definition}

The upshot is that we have a good understanding of the ordinary fiber. Specifically, Lemma \ref{lemma-ordinary-locus} shows that the ordinary locus of $\K(G,R)$, endowed with the subspace topology, is homeomorphic to $\Spec(R[1/n])$ where $n$ is the order of $G$. Consequently, the focus shifts to understand the modular locus, and  how these pieces glue together. 

The following result is a combination of Corollary \ref{coro-coh-open} and Theorem \ref{thm-closure-R[1/n]}. 

\begin{theorem}
    Let $G$ be a $p$-group and $R$ be a commutative Noetherian ring. Then the support of the acyclic complexes $\K_{\mathrm{ac}}(G;R)$ is closed, and its complement is $V_{G;R}\cong \Spec^h(H^\bullet(G;R))$. 
    Moreover,  the closure of the ordinary locus of $\K(G;R)$ in $\Spc(\K(G;R))$ is contained in the cohomological open $V_{G;R}$.
\end{theorem}

We can describe the modular locus when $p$ is not a zero-divisor in the ground ring. The following result corresponds to Theorem \ref{thm-modular-locus}.

\begin{theorem}
    Let $G$ be a $p$-group, and $R$ be a commutative Noetherian ring such that $p$ is not a zero-divisor. Then the map 
    \[
    \Spc(f_p^\ast)\colon  \Spc(\K(G;R/pR)) \to \Spc(\K(G;R))
    \]
    induced by reduction modulo $p$
    is a closed immersion with image the modular locus of $\K(G;R)$.
\end{theorem}

For rings as above, we obtain a full description of $\Spc(\K(G;R))$ in terms of $\Spc(\K(G;R/pR))$, and $\Spec^h(H^\bullet(G;R))$; see Section \ref{sec: examples} for further details. 

We now focus on commutative Noetherian rings of characteristic $p$. In this case, we have a similar reduction to elementary abelian $p$-groups as in \cite{BG23b}. The relevant result is presented later in this document as Theorem \ref{colimit theorem}. 

\begin{theorem}
  Let $G$ be a finite $p$-group  and $R$ be a commutative Noetherian ring of characteristic $p$. Then there is a homeomorphism 
  \[\varphi\colon \underset{(H,K)\in \mathcal{E}(G)}{\colim}  \mathrm{Spc}(\K(H/K;R))\to \mathrm{Spc}(\K(G;R))  \] where $\mathcal{E}(G)$ denotes the category of elementary abelian $p$-sections. 
\end{theorem}

With the preceding reduction in place, we now focus our attention on elementary abelian $p$-groups. A key tool in the work of Balmer and Gallauer is the so-called \emph{twisted cohomology ring}, which yields a precise description of the spectrum of permutation modules over elementary abelian $p$-groups via the Zariski spectra of certain localizations of this ring. We extend their construction of the twisted cohomology ring to a more general Noetherian base (see Convention \ref{convention on R}) and show that it retains many of the desirable properties known in the field case. In particular, we prove that it remains finitely generated, among other features. These results are developed in Sections \ref{section:invertible objects}, \ref{sec: twisted}, and \ref{sec: comparing}.

It is worth highlighting that the restriction we impose on the ring $R$ when constructing the twisted cohomology ring is purely technical, not conceptual.

We are now prepared to state our main result regarding the topology of the spectrum of permutation modules over elementary abelian $p$-groups and its connection with twisted cohomology. This generalizes \cite[Theorems 10.5 and 10.6]{BG23b}. Further details are given in Corollary \ref{cor: spc is a dirac scheme} and Corollary \ref{cor: comp is open immersion}.

\begin{theorem}
Let $E$ be an elementary abelian $p$-group and $R$ be a commutative Noetherian ring such that $\mathrm{ann}_R(p)=R$ or $\mathrm{ann}_R(p)=0$ (e.g., $R=\mathbb{Z}$). Then the comparison map 
\[
\mathrm{Comp}\colon \Spc(\K(E;R))\to \Spec^h(H^{\bullet\bullet}(E;R))
\]
is an open immersion. Here, $H^{\bullet\bullet}(E;R)$ denotes the twisted cohomology ring of $E$.  Moreover,  there is an open cover $\{U(H)\}_{H\leq E}$ of $\mathrm{Spc}(\K(E;R))$ such that the sheaf of  $\mathbb Z$-graded rings obtained from the graded endomorphism rings of the monoidal unit in $\K(E;R)|_{U(H)}$, together with $\Spc(\K(E;R))$, forms a Dirac scheme.  
\end{theorem}

\subsection*{Examples} We apply the techniques developed in this work, together with results from \cite{BG23b}, to study several explicit examples over the integers, presented in Section \ref{sec: examples}.

In particular, our methods provide a full picture over commutative ground rings such that the order of the group is not a zero-divisor; see Example \ref{ex-main}. 

We highlight the case of a cyclic $p$-group over $\mathbb{Z}$. In this setting, the modular locus is well understood thanks to \cite[Section 8]{BG23b}, and the ordinary locus is also accessible, as it corresponds to $\Spec(\mathbb{Z}[1/p])$.  Let us now present $\Spc(\K(C_{p^n};\mathbb Z))$:

\begin{equation*}
\kern2em\vcenter{\xymatrix@C=.0em@R=.4em{
{\color{Brown}\overset{(2)}{\bullet}}    \ar@{-}@[Brown][rrrrrrrrrrrrrrrrdddd]
&& {\color{Brown}\overset{(3)}{\bullet}}  \ar@{-}@[Brown][rrrrrrrrrrrrrrdddd]
&& \cdots
&& {\color{OliveGreen}\overset{\mm_0}{\bullet}} \ar@{-}@[OliveGreen][rdd] 
&& {\color{OliveGreen}\overset{\mm_1}{\bullet}} \ar@{-}@[OliveGreen][ldd] \ar@{-}@[OliveGreen][rdd] 
&& 
&& {\color{OliveGreen}\overset{\mm_{n-1}}{\bullet}} \ar@{-}@[OliveGreen][ldd] \ar@{-}@[OliveGreen][rdd] 
&& {\color{OliveGreen}\overset{\mm_n}{\bullet}} \ar@{-}@[OliveGreen][ldd] \ar@{~}@[Gray][rrdddd] 
&& {\color{OliveGreen}\cdots}
&& {\color{Brown}\overset{(q)}{\bullet}}  \ar@{-}@[Brown][lldddd]
&& \cdots 
\\ \\ 
&&  
&& 
&&
&  {\color{OliveGreen}\underset{\pp_1}{\bullet}} & 
&& \cdots
&& 
& {\color{OliveGreen}\underset{\pp_n}{\bullet}}&  
&&  
&&  
\\ \\ 
&&
&& 
&& 
&& 
&& 
&& 
&& 
&& {\color{Brown}{\bullet_{(0)}}}   
&& 
}}\kern-1.11em
\end{equation*}
the green region corresponds to the modular locus, while the brown region represents the ordinary locus. The points $\mm_i$ and $\pp_j$ are those identified by Balmer and Gallauer under the map induced by the base change functor $\iota_{(p)}^\ast$.

\subsection*{Acknowledgments}
 This work was supported by the Deutsche Forschungsgemeinschaft (Project-ID 491392403 – TRR 358).
UVD would like to thank Henning Krause for inviting him to Bielefeld University in November 2023 and March–April 2025, during which part of this work was carried out. He also acknowledges the support of ANRF/ARGM/2025/002470/MTR, Harish-Chandra Research Institute (HRI), Prayagraj, and Homi Bhabha National Institute (HBNI), Mumbai.

\subsection*{Notation and Conventions} We follow the standard convention of using `tt’ as an abbreviation for either
`tensor-triangulated’ or `tensor-triangular’, depending on the context. Throughout this work, we assume familiarity with the basics of tt-geometry and freely use notation and terminology from \cite{Bal05,Bal10,stevenson2018tour}. Furthermore, we use the term `geometric functor' to refer specifically to a coproduct-preserving tt-functor between rigidly-compactly generated tt-categories.

Let us now highlight some of the most frequently used notations in this text:

\begin{enumerate}
\item We write $\Spc(-)$ for the Balmer spectrum functor, which is defined on essentially small tt-categories and tt-functors.
\item For a graded-commutative ring $R$, we denote its homogeneous spectrum by $\Spec^h(R)$.
\item Calligraphic symbols such as $\mathcal{P}, \mathcal{Q}, \ldots$ are used to denote prime ideals in arbitrary tt-categories.
\item Fraktur symbols such as $\pp, \qq, \ldots$ denote homogeneous prime ideals in a graded ring $R$, while $\aa, \bb, \ldots$ are reserved for \textit{cohomological primes}, i.e., prime ideals in the Balmer spectrum $\Spc(\mathbf{D}_b(\mathrm{mod}(G,R)))$.
\item For a prime ideal $\pp \in \Spec(R)$, we write $k(\pp)$ to denote the residue field of $R$ at $\pp$.
\item The notation $\iota_\pp^\ast$ is used for the functor induced by base change along the residue field $R \to k(\pp)$.
\end{enumerate}

In our treatment of complexes, we adopt homological grading. We also follow standard group-theoretic conventions: for groups $K, H, G$, we write $H^g$ to mean $g^{-1}Hg$, ${}^gH$ to mean $gHg^{-1}$, and $K \leq_G H$ to denote $G$-subconjugation, i.e., $K^g \leq H$ for some $g \in G$.


\section{Preliminaries}


Throughout this text, we will work exclusively with finite groups. 

\begin{recollection}
Let $G$ be a group and $R$ be a commutative Noetherian ring.  A permutation $RG$--module $M$ is a $RG$--module isomorphic to the $R$--linearization of a $G$--set $X$. In other words, a permutation $RG$--module is a module with an $R$--basis that is also a $G$--set. In fact, any permutation $RG$--module is a coproduct of modules of the form $R(G/H)$ where $H$ is a subgroup of $G$.

We write $\Perm(G;R)$ to denote the additive category of permutation $RG$--modules and morphisms $RG$--homomorphisms. We let   $\perm(G;R)$  denote the full subcategory of $\Perm(G;R)$ consisting of finitely generated permutation modules.  In particular, note that $\perm(G,R)$ has countably many isomorphism classes of objects.  In general, the category $\perm(G;R)$ is not idempotent complete; however, if $k$ is a field of positive characteristic $p$ and $G$ is a  $p$-group, then  $\perm(G;k)$ is idempotent complete.      
\end{recollection}

\begin{definition}\label{def:big derived category}
    The \textit{small derived category of permutation $RG$--modules} $\K(G;R)$ is the idempotent completion of the bounded homotopy category $\mathbf{K}_b(\perm(G;R))$. In symbols, 
    \[
    \K(G;R)\coloneqq\mathbf{K}_b(\perm(G;R))^\natural
    \]
    The  \textit{big derived category of permutation $RG$--modules} $\T(G;R)$ is defined as the  localizing subcategory of $\mathbf{K}(\Perm(G;R))$ generated by $\K(G;R)$. 
\end{definition}

\begin{recollection}
  The category $\Perm(G;R)$ is a symmetric monoidal category with the tensor product given by $\otimes_R$ endowed with the diagonal action of $G$. This tensor product descends to $\mathbf{K}(\Perm(G;R))$, and hence equips $\T(G;R)$ with a symmetric monoidal structure on which the monoidal unit is given by $R$ viewed as a complex concentrated in degree 0. In fact, 
the big derived category of permutation $RG$--modules $\T(G;R)$ is a rigidly-compactly generated tensor triangulated category, and $\K(G;R)$ sits inside $\T(G;R)$ as the rigid-compact part. We refer to \cite[Section 3]{BG21} for further details.  
\end{recollection}

\begin{remark}
    When $G$ is the trivial group, we have that $\T(1;R)=\mathbf{D}(R)$ and $\K(1;R)=\mathbf{D}_\textrm{perf}(R)$.
\end{remark}

\begin{recollection} \label{Db is a localization of KG}
  In \cite[Section 4]{BG22b}, Balmer and Gallauer showed that there is a localization functor  
  \[\bar{\Upsilon}^G\colon\T(G;R)\to \mathbf{K}_\textrm{perm}(\mathrm{Inj}(RG))\] 
  where the right-hand side category is the localizing subcategory of $\mathbf{K}( \mathrm{Inj}(RG))$ generated by Tate resolutions of permutation $RG$--modules. This functor,  on compacts, gives a localization functor 
  \[
  \Upsilon^G\colon\K(G;R)\to \mathbf{D}_\mathrm{perm}(G;R)
  \]
  where the right-hand side category denotes the thick subcategory of $\mathbf{D}_b(RG)$ generated by the finitely generated permutation $RG$-modules. In particular, when $R$ is regular, there are equivalences $\mathbf{K}_\textrm{perm}(\mathrm{Inj}(RG))\simeq \mathbf{K}( \mathrm{Inj}(RG))$ and $\mathbf{D}_\mathrm{perm}(G;R)\simeq \mathbf{D}_b(RG)$.
\end{recollection}

\begin{remark}\label{Def of coho primes}
    Since $\Upsilon^G$ is a localization, we obtain that it induces a topological embedding on Balmer spectra 
    \[
    \Spc(\mathbf{D}_\textrm{perm}(G;R))\to \Spc(\K(G;R))
    \]
    and one often refers to the points in $\Spc(\mathbf{D}_\textrm{perm}(G;R))$ as \textit{cohomological primes}. In fact, when $R$ is a field, the above map has open image by \cite[Proposition 3.22]{BG23b}. We will extend this result to more general commutative Noetherian ground rings in Corollary \ref{coro-coh-open}, hence we will refer to the image of the previous functor as the \textit{cohomological open} of $\Spc(\K(G;R))$. 
\end{remark}

\begin{remark}\label{functoriality of K}
    Both categories $\K(G;R)$ and $\T(G;R)$ depend functorially on $G$ and $R$. It is contravariant in the group, and covariant in the ground ring. That is, for a group homomorphism $\alpha\colon G\to G'$ we obtain a geometric tt-functor 
    \[ \alpha^\ast \colon \T(G';R)\to \T(G;R)\] 
    and we follow the usual convention regarding notation and terminology when $\alpha$ is an inclusion, surjection, or conjugation. That is, depending on such properties, we may refer to the functor $\alpha^\ast$ as restriction, inflation, or conjugation. On the other hand, for a map of commutative rings $f\colon R\to S$,  base change along $f$ gives a coproduct preserving tensor functor $\Mod(RG)\to \Mod(SG)$. This functor sends (finitely generated) permutation modules to (finitely generated) permutation modules. In particular, we obtain a geometric tt-functor
    \[f^\ast \colon \T(G;R)\to \T(G;S).\]
    Finally, for a prime $\pp$ in $R$, we let $\iota_\pp$ denote the residue field $R\to k(\pp)$. And hence, write $\iota_\pp^\ast\colon \T(G;R)\to \T(G;k(\pp))$ to denote the corresponding base change functor.
\end{remark}

We conclude this section by recalling the following result, which corresponds to \cite[Proposition 3.8]{BG23b}. 

\begin{proposition}\label{Spc is a point in ordinary char}
    Let $k$ be a field and $G$ be a group such that $|G|$ is invertible in $k$. Then $\Spc(\K(G;k))$ is a point. 
\end{proposition}

\section{Recollections, base change and modular fixed points}

Throughout, we let $R$ denote a commutative Noetherian ring, and $G$ a finite group. In this section, we recall some relevant results from the literature that will play a role in this work. 

\subsection{Fiberwise building} Let $\pp$ be a prime of $R$. Recall that $\iota_\pp^\ast\colon \T(G;R)\to \T(G;k(\pp))$ denotes the geometric functor induced by base change along the residue field $R\to k(\pp)$. The following result is a combination of \cite[Theorem 4.8]{Gom25}, and \cite[Corollary 6.12]{Gom25} and its proof. 

\begin{theorem}\label{thm-weakly descent}
    Let $G$ be a finite group and $R$ be a commutative Noetherian ring. Then the family of functors 
    \[
    \left( \iota_\pp^\ast\colon \T(G;R)\to \T(G;k(\pp))\right)_{\pp\in \Spec(R)}
    \]
     is jointly conservative.
     Moreover, $\Spc(\K(G;R))$ is a Noetherian space, and the above family of functors induces a continuous surjection 
    \[
    \bigsqcup_{\pp\in \Spec(R)}\Spc(\K(G;k(\pp)))\to \Spc(\K(G;R)).
    \]
\end{theorem}

\subsection{Reduction to $p$-groups} The following result allows us to focus on $p$-groups since they determine the topology of the Balmer spectrum for general finite groups. 

\begin{theorem}
    \label{thm-red-p-groups}
    Let $G$ be a finite group. Let  $\mathcal{E}$ denote the family of $p$-power order subgroups of $G$ and the trivial subgroup, where $p$ runs over all  primes dividing the order of $G$. Then there is a coequalizer of topological spaces  
       \[
      \begin{tikzcd}
         \displaystyle  \bigsqcup_{\substack{H,K\in \mathcal{E}\\ [g]\in H\backslash G/K}} \Spc(\K(H\cap{}^gK;R)) \arrow[r,shift left=1] \arrow[r,shift left=-1] & \displaystyle \bigsqcup_{H\in \mathcal{E}}\Spc(\K(H;R)) \arrow[r,""] & \Spc(\K(G;R))
      \end{tikzcd}
     \]
     where the left arrows are induced by the inclusion and by conjugation. 
\end{theorem}

\begin{proof}
    See \cite[Lemma 6.1]{Gom25}.
\end{proof}

\subsection{The modular locus} Let $G$ denote a $p$-group. We now set some terminology. 

\begin{definition}
    \label{def-modular-locus}
   Let $\eta\colon \mathbb Z \to R$ be the structural map. The \textit{modular locus of $\K(G;R)$} is the subset 
   \[
   \Spc(\K(G;R))_p\coloneqq (\Spec(\eta)\circ \mathrm{Comp})^{-1}(p)\subseteq \Spc(\K(G;R))
   \]
   its \textit{ordinary locus} is the set-theoretic  complement of the modular locus. 
\end{definition}

\begin{remark}
    Note that the modular locus is a closed subset, and hence the ordinary locus is an open subset. 
\end{remark}

The terminology is motivated by the following result. 

\begin{lemma}
    \label{lemma-ordinary-locus}
    If the order of $G$ is invertible in $R$, then Balmer's comparison map \cite{Bal10} 
    \[
     \rho\colon \mathrm{Spc}(\K(G;R))\to \mathrm{Spec}(R)
    \]
    is a homeomorphism. Moreover, for an arbitrary commutative ring $R$, the ordinary locus of $\K(G;R)$ identifies with $\Spec(R[1/|G|])$.
\end{lemma}

\begin{proof}
    The first part is \cite[Proposition 6.5]{Gom25}. For the second part, consider the base change functor 
    \[
    R[1/|G|]\otimes _R -\colon \K(G;R) \to \K(G;R[1/|G|]).
    \]
    This map is essentially surjective by \cite[Lemma 2.24]{BG22}, and hence a topological embedding on Balmer spectra. The proof of \cite[Proposition 6.6]{Gom25} gives us that the image of the previous map is precisely the ordinary locus of $\K(G;R)$ as we wanted. 
\end{proof}

\subsection{Generalized modular fixed points} We now recall generalized modular fixed points; we refer to \cite{BG23b}, \cite{Gom25} and \cite{GM26a} for further details. 

\begin{recollection}
    \label{rec-gen-modular-fixed}
    Let $R$ be a commutative Noetherian ring of characteristic $p$, $G$ be a $p$-group, and $N\unlhd G$ be a normal subgroup. Let $\mathcal{F}_N=\{H\leq G\mid N\not\leq H\}$, and set 
\[
\mathrm{proj}(\mathcal{F}_N)= \mathrm{add}^\natural\{R(G/H)\mid H\in \mathcal{F}_N\}.
\]
    The composite 
    \[
   \mathrm{perm}(G/N;R)^\natural\xrightarrow[]{\mathrm{Infl}} \mathrm{perm}(G;R)^\natural\xrightarrow{\mathrm{quot}} \mathrm{perm}(G;R)^\natural/\mathrm{proj}(\mathcal{F}_N)
    \]
    is an equivalence of tensor additive categories. We obtain a tt-functor   
\[
\Psi^N \colon \K(G;R)\xrightarrow{} \K(G/N;R)
\]
by applying the previous equivalence  degree-wise. Now, let $H\leq G$ be a subgroup. The generalized modular $H$-fixed points functor is the tt-functor given as the composite 
\[
\Psi^H\colon \K(G;R) \xrightarrow[]{\mathrm{Res}}\K(N_G(H);R)\xrightarrow[]{\Psi^{N_G(H),H}}\K(\WGH;R).
\]
For an arbitrary commutative ring $R$, we define the generalized modular $H$-fixed points by 
\[
\Psi^H_p \colon \K(G;R) \xrightarrow[]{R/pR\otimes_R -} \K(G;R/pR ) \xrightarrow[]{\Psi^H} \K(\WGH;R/pR).
\]
We refer to \cite[Section 6]{Gom25}, and \cite[Section 4]{GM26a} for further details. 
\end{recollection}

We conclude this section by recording the following result which corresponds to \cite[Proposition 6.9]{Gom25}.

\begin{proposition}
\label{prop-basechange-psi}
Let $\pp\in \mathrm{Spec}(R)$ such that $k(\pp)$ has characteristic $p$, and $H$ be a subgroup of $G$. Then we have a commutative diagram 
\[
\begin{tikzcd}
     \K(G;R) \arrow[r,"\Psi^H_p"] \arrow[d] & \K(\WGH;R/pR) \arrow[d]\\
     \K(G;k(\pp))\arrow[r,"\Psi^H"] &  \K(\WGH;k(\pp))
\end{tikzcd}
\]
where the vertical arrows are induced by base change to the residue fields, and the bottom map is the  modular $H$-fixed points functor from \cite{BG23b}.
\end{proposition}


\section{Koszul objects and their support}\label{sec: Kosz}

Throughout this section, we let $G$ be a  $p$-group and let $R$ be a commutative Noetherian ring. We will introduce Koszul objects which are analogous to those defined in \cite{BG23b}, and we will describe their support.  These Koszul objects play a crucial role in the set-theoretic description of the Balmer spectrum of $\K(G;R)$. More concretely, for a subgroup $H$ of $G$ we describe a bounded complex of finitely generated permutation $RG$--modules $\widetilde \kos_G(H;R)$ which, in particular, generates the kernel of the restriction functor $\mathrm{Res}^G_H$ as a tt-ideal. Moreover, this construction is well behaved under base change.

For odd primes, it turns out that tensor-induction is still good enough to produce such objects. However,  for $p=2$, additional work is required. In this case, we produce Koszul objects by refining the resolutions constructed in the proof of \cite[Theorem 3.1]{BG22}. Let us begin with some generalities on tensor-induction.

 Let us briefly recall tensor-induction for complexes. 

\begin{recollection}\label{tensor-induction}
Let $H$ be a normal subgroup of $G$ of index $n$.  Let $g_1,\ldots, g_n\in G$ be a set of representatives of the cosets $G/H$. Using  the bijection $\{1,\ldots,n\}\to G/H, \, i\mapsto [g_i]$, we get that the $G$-action on $G/H$, by multiplication on the left, induces a group homomorphism $\sigma \colon G\to S_n$. Now, consider the action of $S_n$ on $H^n$ by permuting the factors, and the associated semi-direct product $S_n\ltimes H^n$.  This allows us to define an injective group homomorphism by
\[ i\colon G\to S_n \ltimes H^n, \, g\mapsto (\sigma g, h_1, \ldots, h_n)\]
where $g\cdot g_j =g_{(\sigma g)(j)} \cdot h_j$, for all $j=1,\ldots, n$. 
Now, for a bounded complex $x$ of $RH$--modules we write  
\[{}^\otimes \Ind_H^G(x)\coloneqq i^\ast (x^{\otimes n})\] where $i^\ast$ denotes the restriction functor along $i$. Here the $S_n$-action on $x^{\otimes n}$ involves signs according to the Koszul rule. See \cite[Recollection 3.5]{BG22} for further details. 
\end{recollection}

\begin{remark}
    In the modular case, Balmer and Gallauer defined the so-called Koszul objects $\kos_G(H)$ in the derived category of permutation modules for any subgroup $H$ of $G$ using tensor-induction. Namely, 
    \[
    \kos_G(H)={}^\otimes\Ind_H^G(0\to R\to R\to 0)
    \]
    However, when $R$ is not a field, some signs appear. In particular, ${}^\otimes\Ind_H^G(0\to R\to R\to 0)$ may no longer be a complex of permutation modules, at least for $p=2$. Let us explain how to circumvent this problem. 
\end{remark}

\subsection*{Case $p=2$.} Let $G$ be a 2-group and let $R$ denote a commutative Noetherian ring.

\begin{remark}
   Recall that a $RG$--module $M$ is a sign-permutation module if it has an $R$--basis $X$, such that $G\cdot X\subseteq X\cup (-X)$. 
\end{remark}

\begin{recollection}\label{sign representation}
    Let $H$ be a subgroup of $G$ of index 2. Let $R_{\mathrm{sgn}}$ denote the one-dimensional sign representation of $G/H$, and write $L$ to denote  $\Infl^{G}_{G/H}( R_{\mathrm{sgn}})$. Note that there is a quasi-isomorphism of $RG$--modules which is obtained by inflation: 
\begin{center}
\begin{tikzcd}
\tilde{L}  \arrow[d, " s_H"']  \coloneqq & & \cdots 0 \arrow[r] & R \arrow[r] \arrow[d]  & R(G/H) \arrow[r] \arrow[d] & 0 \\
L \coloneqq & & \cdots 0 \arrow[r] & 0 \arrow[r] & L \arrow[r] & 0
\end{tikzcd}   
\end{center}
    where $L$ is concentrated in degree 0. Note that for any $K\leq H$, the restricted module $\Res^G_K(L)$ is simply the trivial representation, and $\Res^G_K(R(G/H))$ is simply two copies of the trivial module $R$. Hence $\Res^G_K(s_H)$ becomes a map of complexes of permutation modules. In fact, it becomes a homotopy equivalence. We will record this observation in the following lemma. 
\end{recollection}

\begin{lemma}\label{sH is iso}
    With the same notation as in Recollection \ref{sign representation}, we obtain that the map $\Res^G_K(s_H)$ is an isomorphism in $\K(K;R)$ for any $K\leq H$. 
\end{lemma}
\begin{proof}
    We already observed that the complexes in question are complexes of permutation modules and have trivial $K$--action. Moreover, the terms of both complexes are free as $R$--modules. Hence, the cone of $\mathrm{Res}^G_K(s_H)$ is split exact and the result follows. 
\end{proof}

\begin{lemma}\label{lem:Res_vanishing}
    Let $H$ be a subgroup of $G$ of index 2.  Let $x$ be a bounded complex of permutation $RH$--modules. Then the following properties hold. 
    \begin{enumerate}
        \item Let $y$ denote the tensor-induction ${}^\otimes\Ind_H^G(x)$, which is a bounded complex of sign-permutation $RG$--modules but not necessarily of permutation $RG$--modules. If $\Res^H_{H'}(x)$ is contractible for $H'\leq H$, then the complex  $\Res^{G}_K(y)$ is a contractible bounded complex, for any $K\leq H'$. 
        \item Assume that $x$ is concentrated in non-negative degrees, and  $x_0=R$ and $x_1=\Ind_K^H(z)$ for some $RK$--module $z$, and some $K\leq H$.   Then ${}^\otimes\Ind_H^G(x)_0=R$ and ${}^\otimes\Ind_H^G(x)_1=\Ind_K^G(z)$. 
    \end{enumerate}
\end{lemma}

\begin{proof}
    The first part follows from the Mackey formula for tensor-induction:
    \begin{eqnarray*}
     \Res^{G}_K({}^\otimes\Ind^{G}_{H}(x)) \simeq & \displaystyle\bigotimes_{h \in H \backslash G/K} {}^\otimes \Ind^H_{H \cap {}^hK} ({}^h \Res^{H}_{H \cap {}^hK} (x)). 
    \end{eqnarray*}
    In particular, one of the factors on the right hand side is a restriction of our complex of $RH$--modules $x$, which is already contractible by hypothesis. 

    The second assertion follows from the definition of the tensor-induction. In particular, ${}^\otimes\Ind_H^G(x)_1$ consists of two copies of $\Ind^H_K(z)$ and the action of $G$ will permute those copies. Hence, it can be identified with $\Ind_K^G(z)$ as we wanted. 
\end{proof}

\begin{lemma}\label{lem:sgn_modification}
    Let $H$ be a subgroup of $G$ of index 2.  Let $x$ be a non-negative bounded acyclic complex of  sign-permutation $RG$--modules satisfying: 
    \begin{enumerate}
        \item $\Res^G_H(x)$ is a complex of permutation $RG$--modules.  
       \item $x_0$ is the trivial module $R$. 
       \item $x_1$ is of the form $\Ind^G_K(z)$ for some $RK$--module $z$, and $K\leq H$. 
       \item $\Res^G_{H'}(x)$ is contractible for some $H'\leq H$.
    \end{enumerate}
    We can inductively modify $x$ in finitely many steps to get a non-negative bounded acyclic complex $\tilde{x}$ of  permutation 
$RG$--modules satisfying:
\begin{enumerate}
    \item $\tilde{x}_0$ is the trivial module $R$. 
       \item $\tilde{x}_1$ is of the form $\Ind^G_K(z')$ for some $RK$--module $z'$.
       \item $\Res^G_{H'}(\tilde{x})$ is contractible.
\end{enumerate}
\end{lemma}

\begin{proof}
    This is essentially contained in the proof of  \cite[Theorem 3.1]{BG22}. Hence, the task is to verify that such a modification of the complex $x$ satisfies the required properties. Let $L$ be the sign representation and consider the map $s_H$ from Recollection \ref{sign representation}.  The complex $x$ can be modified inductively; say $x^{m+1}$ is such a way that $x^{ m+1}_j$ is  a permutation $RG$--module for any $j \geq m+1$. Additionally $x^{ m+1}_0$ is either $L$ or $R$ and $x^{ m+1}_1$ is a module of the form $\Ind_K^G(z)$ for some $RK$--module $z$. 
    
    The modification is done using descending induction. Since the complex is bounded, the conditions on $x$ will hold for some $m+1$ large enough. Hence, fix $m+1$ satisfying the previous conditions. 
    Define $x^{ m} \coloneqq \mathrm{cone}(s_H \otimes t^m)$ where $t^m$ is a morphism of complexes defined as:
\begin{center}
\begin{tikzcd}
x^{m+1'}  \arrow[d, " t^m"']  \coloneqq &  \cdots  x^{m+1}_{m+2} \arrow[r] \arrow[d] & x^{ m+1}_{m+1} \arrow[r] \arrow[d]  & (x^{ m+1}_m)^{+} \arrow[r] \arrow[d] & 0 \arrow[d]  \cdots \\
x^{m+1 ''} \coloneqq & \cdots   0  \arrow[r] & L \otimes (x^{m+1}_m)^{-} \arrow[r] &  x^{ m+1}_{m - 1} \arrow[r] &  x^{ m+1}_{m - 2} \cdots
\end{tikzcd}   
\end{center}
here we are using the decomposition of $RG$--modules
\[ 
x^{ m+1}_m=(x^{ m+1}_m)^{+} \oplus (L \otimes (x^{m+1}_m)^{-})
\]
given in Lemma 3.9 of \cite{BG22}; in particular both $(x^{ m+1}_m)^{+}$ and $(x^{ m+1}_m)^{-}$ are permutation $RG$--modules. We can check that  $x^{m}_j$ is a permutation $RG$--module for any $j \geq m$ and $x^{m}$ has the required properties, which completes the induction step. Indeed, since $L\otimes L\simeq R$ and $\Res^G_K(L)\simeq R$, we obtain that  $x^m_m$ is a permutation module, $x^m_0$ is either $R$ or $L$ and $x^m_1$ is still a module of the form $\Ind_{K}^G(z)$ for some $z$ by the projection formula. 

Finally, we need to verify this induction step preserves the condition $(3)$. But this follows by Lemma \ref{sH is iso} and our assumption about the contractibility of  $x^{m+1}$ when restricted to $H'\leq H$. That is, $\Res^G_H(s_H\otimes t^m)$ is a homotopy equivalence. Thus, the result follows. 
\end{proof}

\begin{definition}\label{def sign mod}
The modification in the above lemma will be referred to as \emph{sign-modification}.
\end{definition}

\begin{construction}\label{kos in char 2}
    Let $H$ be a subgroup of $G$. We will construct a complex $\widetilde{\kos}_G(H;R)$ of permutation $RG$--modules with the following properties: 
    \begin{enumerate}
        \item $\widetilde{\kos}_G(H;R)_0=R$.
        \item $\widetilde{\kos}_G(H;R)_1$ is of the form $\Ind_H^G(x)$ for some $RH$--module $x$. 
        \item $\mathrm{Res}^G_H(\widetilde{\kos}_G(H;R))\simeq 0$ in $\K(H;R)$.
    \end{enumerate}
The construction is inductive on the index of $H$ in $G$. Consider a normal filtration
\[
H_0 = H \trianglelefteq H_1 \trianglelefteq H_2 \trianglelefteq \ldots \trianglelefteq H_l=G
\]
where $H_i/H_{i-1}$ has order $2$, for any $i=1,\ldots,l$. Now, let $y^0$ denote the complex 
\[
\ldots \to 0\to R\xrightarrow[]{1} R\to 0\to \ldots
\]
of $kH_0$--modules, where the trivial module $R$ is in degree 0 and 1. For $i>0$, we define $y^i$ as the sign-modification (see Definition \ref{def sign mod}) of the complex 
${}^\otimes\Ind_{H_{i-1}}^{H_i}(y^{i-1})$. We let $\widetilde{\kos}_G(H)\coloneqq y^l$. Since the base of the inductive definition satisfies the desired properties, so does $\widetilde{\kos}_G(H)$ by Lemma \ref{lem:sgn_modification}. In particular, let us stress that $y^1$ does not need any sign modification. 
\end{construction}

\subsection*{General case.} Let $G$ be a $p$-group. We already mentioned that tensor-induction for odd primes is good enough to produce Koszul objects. Let us record this in a uniform fashion.

\begin{definition}\label{notation for kos}
    Let $H$ be a subgroup of $G$. We write $\widetilde{\kos}_G(H;R)$ to denote either ${}^\otimes\Ind_H^G(0\to R\to R\to 0)$ if $p>2$; or  the complex $\widetilde{\kos}_G(H;R)$ from Construction \ref{kos in char 2} if $p=2$. We will refer to these objects as \textit{Koszul objects}.
\end{definition}

\begin{remark}\label{three conditions satisfied by kos}
    Note that $\widetilde{\kos}_G(H;R)$ is an acyclic complex of permutation $RG$--modules satisfying that $\widetilde{\kos}_G(H;R)_0=R$ and $\widetilde{\kos}_G(H;R)_1$ is module of the form $\Ind_H^G(z)$ for some $RH$--module $z$. Indeed,  for $p>2$, by \cite[Lemma 3.8]{BG22} we know that any sign-permutation module is a permutation module. Moreover, since the complex $0\to R\to R\to 0$ is acyclic, its tensor-induction remains acyclic. Finally, the description of the module in degree 0 and degree follows from the definition of tensor-induction, and it is similar to \cite[Proposition 3.15]{BG23b}. For $p=2$, this follows from construction (see Construction \ref{kos in char 2}).
\end{remark}

We are ready to describe one of the most relevant properties of Koszul objects. In fact, it would be a consequence of the construction and an abstract result \cite[Corollary 3.20]{BG23b}. Let us include the statement here for the convenience of the reader. 

\begin{lemma}\label{conditions to generated a given ideal}
    Let $\A$ be a rigid tensor additive category and $\mathcal{I}$ be a thick tensor ideal of $\mathbf{K}_b(\A)$. Let $x$ be a bounded complex in $\mathcal{I}$ satisfying the following properties. 
    \begin{enumerate}
        \item $x$ is concentrated in non-negative degrees. 
        \item The  thick tensor ideal generated by  $x_0$ contains $\mathcal I$.
        \item $x_1\otimes y\simeq 0$ in $\mathbf{K}_b(\A)$ for any $y\in \mathcal I$.
    \end{enumerate}
    Then $x$ generates $\mathcal{I}$ as a thick tensor ideal. 
\end{lemma}

\begin{corollary}\label{supp of kos}
  Let $H \leq G$. Consider the restriction functor $\Res^G_H\colon \K(G;R)\to \K(H;R)$. Then $\widetilde{\kos}_G(H; R)$ generates $\ker(\Res^G_H)$ as a thick tensor ideal. In symbols,  
  \[
  \ker(\Res^G_H) = \langle \widetilde{\kos}_G(H; R)\rangle.
  \]
  Equivalently, $\mathrm{supp}(\widetilde{\kos}_G(H; R))=\mathrm{supp}(\ker(\Res^G_H))$ in $\Spc(\K(G;R))$. In particular, $\K_{ac}(G; R) = \ker(\Res^G_1) = \langle \widetilde{\kos}_G(1, R) \rangle$.
\end{corollary}

\begin{proof}
This follows from Lemma \ref{conditions to generated a given ideal}. Indeed, the required three properties of $\widetilde{\kos}_G(H; R)$ follow by Remark \ref{three conditions satisfied by kos}.
\end{proof}

\begin{corollary}\label{coro-coh-open}
    The localization functor $\Upsilon^G\colon \K(G;R)\to \mathbf{D}_{\mathrm{perm}}(G;R)$ induces an open inclusion on Balmer spectra. The complement of its image is the support of $\K_{ac}(G;R)$. 
\end{corollary}

Let us conclude this section with an observation on the behavior of the Koszul objects under base change along the reside fields of $R$. We need some preparation. 

\begin{recollection}\label{restriction under base change}
    Let $H$ be a subgroup of $G$, and let $\pp\in \mathrm{Spec}(R)$. Then we have a commutative square 
    \begin{center}
        \begin{tikzcd}
               \K(G;R) \arrow[r,"\iota_\pp^\ast"] \arrow[d,"\mathrm{Res}_{H,R}"'] & \K(G;k(\pp)) \arrow[d,"\mathrm{Res}_{H,k(\pp)}"] \\
               \K(H;R) \arrow[r,"\iota_\pp^\ast"]  & \K(H;k(\pp)).
        \end{tikzcd}
    \end{center}
    In particular, we have that $\iota_\pp^\ast(\mathrm{Ker}(\mathrm{Res}_{H,R}))\subseteq \mathrm{Ker}(\mathrm{Res}_{H,k(\pp)})$. We will drop the base ring from the notation if the context is clear. 
\end{recollection}

\begin{corollary}\label{supp of kos under base ch}
Let $\pp\in \mathrm{Spec}(R)$. Then $\iota_\pp^\ast(\widetilde{\kos}_G(H;R))$ generates the same thick tensor ideal in $\K(G;k(\pp))$ as $\mathrm{kos}_G(H;k(\pp))$.    
\end{corollary}

\begin{proof}
Note that the previous statement is trivial if the characteristic of $k(\pp)$ is not $p$. Then assume otherwise.  We claim that $\iota_\pp^\ast(\widetilde{\kos}_G(H;R))$ generates $\ker(\Res^G_H)$, and hence the result will follow by \cite[Proposition 3.21]{BG23b}. But this is a consequence of a second layer of Lemma \ref{conditions to generated a given ideal}. Indeed,  $\iota_\pp^\ast(\widetilde{\kos}_G(H;R))\in \ker(\Res^G_H)$ by Recollection \ref{restriction under base change} and satisfies the following conditions: $\iota_\pp^\ast(\widetilde{\kos}_G(H;R))_0=R\otimes_R k(\pp)\cong k(\pp)$ and $\iota_\pp^\ast(\widetilde{\kos}_G(H;R))_1=\Ind_H^G(z)\otimes_R k(\pp)\cong \Ind^G_H(z\otimes_R k(\pp))$ (see Definition \ref{notation for kos} and Remark \ref{three conditions satisfied by kos}). Then  $\iota_\pp^\ast(\widetilde{\kos}_G(H;R))$ generates  $\ker(\Res^G_H)$. 
\end{proof}


\section{The Balmer spectrum as a set}


In this section, we will describe the points of the Balmer spectrum of $\K(G;R)$ for $G$ a $p$-group and $R$ a commutative Noetherian ring such that $R/pR\not =0$, otherwise, we have a full description by Lemma \ref{lemma-ordinary-locus}. We begin by setting some notation.

\begin{notation}
  For each subgroup $H$ of $G$, consider the composite map
  \[
  \check \psi^{G,H}_p\coloneqq \check \psi^{H}_p\colon V_{\WGH,R/pR}\xrightarrow{\Spc(\Upsilon^{\WGH})} \Spc(\K(\WGH,R/pR))\xrightarrow[]{\psi^{H}_p} \Spc(\K(G;R)) 
  \]
  where $V_{\WGH,R/pR}$ denotes $\Spc(\mathbf{D}_{\mathrm{perm}}(\WGH;R/pR))$, and $\psi^{H}_p\coloneqq \Spc(\Psi_p^H)$. Recall that $\mathbf{D}_{\mathrm{perm}}(\WGH;R)$ denotes the thick subcategory of $\mathbf{D}_b(R\WGH)$ generated by permutation $R\WGH$--modules, see Recollection \ref{Db is a localization of KG} if needed.  
\end{notation}

Let us record the following result, which is a consequence of  \cite[Theorem 6.11]{Gom25} together with Lemma \ref{lemma-ordinary-locus}.

\begin{theorem}\label{cover of spc KG}
    Let $G$ be a $p$-group and $R$ be a commutative Noetherian ring. Then we have continuous surjection 
  \[
  \bigsqcup_{H\leq G}V_{\WGH;R/pR} \xrightarrow[]{\sqcup \check \psi^{H}} \Spc(\K(G;R))_p
  \]
  where the right hand side denotes the modular locus. 
\end{theorem}

\begin{recollection}\label{Spc of Dperm}
    Recall that an $RG$--lattice is an $RG$--module that is projective as $R$--module. Consider the exact structure on the additive category of finitely generated $RG$--lattices induced by the split exact structure on the category of finitely generated projective $R$--modules. See Section \ref{section:invertible objects} for a further discussion. Let $\mathbf{D}_b(G;R)$ denote the bounded derived category of finitely generated $RG$--lattices associated with this exact structure. This is a tt-category with the monoidal structure induced by $\otimes_R$ with the diagonal action of $G$. By work of Lau \cite{Lau23}, we know that the comparison map  
    \[ \Spc(\mathbf{D}_b(G;R))\to \Spec^h(\End_{\mathbf{D}_b(G;R)}^\ast(R))\]
    is a homeomorphism. Recall that  $\End_{\mathbf{D}_b(G;R)}^\ast(R)$ identifies with the cohomology ring $H^\ast(G;R)$. 

    Now, note that $\mathbf{D}_\mathrm{perm}(G;R)\subseteq\mathbf{D}_b(G;R)$. In fact, they agree when $R$ is regular. In any case, both categories have the same monoidal unit, and their graded endomorphism ring agree.  By the naturality of the comparison map, we get a commutative diagram 
    \begin{center}
        \begin{tikzcd}
             \Spc(\mathbf{D}_b(G;R)) \arrow[r] \arrow[d,"\mathrm{Comp}"']  &  \Spc(\mathbf{D}_\mathrm{perm}(G;R)) \arrow[d,"\mathrm{Comp}"] \\
             \Spec^h(H^\ast(G;R)) \arrow[r,"\cong"] &  \Spec^h(H^\ast(G;R))
        \end{tikzcd}
    \end{center}
    which tells us that the right hand side map is a bijection since the top map is a surjection. Hence it is homeomorphism by \cite[Corollary 2.8]{Lau23} since $\mathbf{D}_\mathrm{perm}(G;R)$ is End-finite. These homeomorphisms will be used to translate certain properties of the spectrum of $\mathbf{D}_\mathrm{perm}(G;R)$ under base change along residue fields of $R$. 
 \end{recollection}

 For convenience, let us record the following result, which appears in the proof of \cite[Theorem 6.11]{Gom25}. 

\begin{proposition}\label{base change and triangular fixed points}
    Let $H$ be a subgroup of $G$, and let $\pp\in \Spec(R)$ such that $k(\pp)$ has characteristic $p$. Then the following diagram is commutative. 
    \begin{center}
        \begin{tikzcd}
        V_{\WGH;k(\pp)}  \arrow[d,"\Spc(\lambda_\pp^\ast)"'] \arrow[r,"\check \psi^{G,H}_\pp"]  & \Spc(\K(G;k(\pp))) \arrow[d,"\Spc(\iota^\ast_\pp)"] \\ V_{\WGH;R/pR} \arrow[r,"\check \psi^{H}_p"]  & \Spc(\K(G;R))
        \end{tikzcd}
    \end{center}
    where $\lambda_\pp^\ast$ is short for the left derived base change functor 
    \[
    -\otimes^L k(\pp)\colon \mathbf{D}_\mathrm{perm}(\WGH;R/pR)\to \mathbf{D}_\mathrm{perm}(\WGH;k(\pp)).
    \]
\end{proposition}

\begin{remark}\label{cohomological primes from residue fields}
    By  \cite{Lau23}, we know that the functors $\lambda_\pp^\ast$ from  the previous proposition, determine a bijection on Balmer spectra: 
    \[
    \bigsqcup_{\pp\in \mathrm{Spec}(R)} V_{G;k(\pp)} \to V_{G;R}
    \] 
    for any finite group $G$ and any commutative Noetherian ring $R$.
\end{remark}

\begin{definition}
     Let $H\leq G$ be a subgroup, and $\aa$ be a \textit{cohomological prime} in  $V_{\WGH;k(\pp)}=\Spc(\mathbf{D}_\textrm{perm}(\WGH);k(\pp))$ for some $\pp\in \Spec(R)$. Define a prime of $\K(G;R)$ as follows. If $\mathrm{char}(k(\pp))=p$, set 
    \[
    \P_{G; R}(H, \aa,\pp) \coloneqq \check \psi^{G,H}_p (\Spc(\lambda_\pp^\ast) (\aa)) = \Spc(\iota^\ast_\pp) (\check \psi^{G,H}_\pp (\aa))\in \Spc(\K(G,  R));
    \] 
   note that the equality comes from the square in Proposition \ref{base change and triangular fixed points}. If $\mathrm{char}(k(\pp))\not= p$, then we define $ \P_{G; R}(H, \aa,\pp)$ as the unique prime of $\K(G;R)$ that goes to $\pp$ under the comparison map; see Lemma \ref{lemma-ordinary-locus}. 
   If the context is clear, we will forget about the subscript and simply write $\P(H,\aa,\pp)$. 
\end{definition}

\begin{remark}\label{all points in K(G) are come from check psi}
    By Theorem \ref{cover of spc KG} combined with Remark \ref{cohomological primes from residue fields}, we get that any point of $\mathrm{Spc}(\K(G;R))$ is of the form $\P(H,\aa,\pp)$ for some $H\leq G$, $\pp\in \Spec(R)$ and $\aa\in V_{\WGH;k(\pp)}$. 
\end{remark}

\begin{corollary}\label{comparison for check psi to residue fields}
     Let $H$ be a subgroup of $G$, and let $\pp\in \Spec(R)$ with $\mathrm{char}(k(\pp))=p$. Then the diagram 
      \begin{center}
        \begin{equation}\label{check psi under base change and comparison}
           \begin{tikzcd}
        V_{\WGH;k(\pp)} \arrow[rr, bend left=25] \arrow[d,"\Spc(\lambda_\pp^\ast))"'] \arrow[r,"\check \psi^{G,H}"]  & \Spc(\K(G;k(\pp)))  \arrow[d,"\Spc(\iota^\ast_\pp)"] \arrow[r] & \Spec(k(\pp))\arrow[d,"\Spec(\iota_\pp)"] \\ V_{\WGH;R/pR} \arrow[r,"\check \psi^{H}_p"] \arrow[rr, bend right=25]  & \Spc(\K(G;R)) \arrow[r] & \Spec(R)
        \end{tikzcd}
    \end{equation}
    \end{center}
    is commutative; here, all unlabeled arrows correspond to Balmer's comparison map. 
\end{corollary}

\begin{proof}
    Follows from Proposition \ref{base change and triangular fixed points} and the naturality of the comparison map.
\end{proof}

\begin{corollary}\label{cohomological primes over different fibers}
    Let $\P(H,\aa,\pp)$ and $\P(K,\aa',\pp')$ be points of $\Spc(\K(G;R))$ for some $H,K\leq G$, $\pp,\pp'\in \Spec(R)$ and $\aa\in V_{\WGH;k(\pp)}$ and $\aa'\in V_{\WGK;k(\pp')}$. If $\pp\not=\pp'$. Then $\P(H,\aa,\pp)\not=\P(K,\aa',\pp')$.
\end{corollary}

\begin{proof}
    This follows from the definition and Corollary \ref{comparison for check psi to residue fields}. Indeed, using the  comparison map from $\Spc(\K(G;R))$ to $\Spec(R)$,  we obtain that $\mathrm{Comp}(\P(H,\aa,\pp))=\pp$ and similarly  $\mathrm{Comp}(\P(K,\aa',\pp'))=\pp'$.
\end{proof}

Recall the Koszul objects from Definition \ref{notation for kos}. We now explain a relation between these objects and the primes $\P(H,\aa,\pp)$.

\begin{lemma}\label{lem:koszul_conjugate_subgroups_ker}
     Let $H$ and $H'$ be subgroups of $G$, $\pp\in \Spec(R)$ and  $\aa \in V_{\Weyl{G}{H'};k(\pp)}$. Suppose that the characteristic of $k(\pp)$ is $p$. Then $\widetilde{\kos}_G(H;R)\in \P_{G; R}(H', \aa,\pp) $ if and only if $H' \leq_G H$. 
\end{lemma}

\begin{proof}
Unpacking the definitions, we get that 
\begin{equation}\label{P under iota}
  \iota^\ast_\pp(\P_{G;R}(H', \aa,\pp))\subseteq \check\psi^{G,H'}(\aa)\eqqcolon \P_{G, k(\pp)}(H', \aa) 
\end{equation}
where the prime $\P_{G, k(\pp)}(H',\aa)$ is precisely  one of those defined in \cite[Definition 7.4]{BG23b}.
If $\widetilde{\kos}_G(H;R)\in \P_{G; R}(H', \aa,\pp)$, then by Equation \ref{P under iota}, we obtain that 
\[
\langle \kos_G (H; k(\pp)) \rangle = \langle \iota_\pp^* (\widetilde{\kos}_{G}(H; R)) \rangle\subseteq \P_{G, k(\pp)}(H',\aa)
\]
where the first equality follows from Corollary \ref{supp of kos under base ch}. Then $H'\leq_G H$ by  \cite[Lemma 7.12]{BG23b}. 

The converse is similar. Indeed,  assume that $H'\leq_G H$.  By definition, 
\[
\P_{G; R}(H', \aa,\pp) = (\iota_\pp^\ast)^{-1}\P_{G;k(\pp)}(H',\aa).
\]
 By Corollary  \ref{supp of kos under base ch}, we have 
 \[
 \iota_\pp^* (\widetilde{\kos}_{G}(H; R))\in \langle \kos_G (H; k(\pp)) \rangle.
 \]
 Another layer of  \cite[Lemma 7.12]{BG23b} gives us that $\kos_G (H; k(\pp))\in \P_{G;k(\pp)}(H',\aa)$ and hence so is $\iota_\pp^* (\widetilde{\kos}_{G}(H; R))$. We conclude that $\widetilde{\kos}_{G}(H; R)\in \P_{G; R}(H',\aa,\pp)$.
\end{proof}

\begin{remark}\label{rem:primes in coprime char}
    Let $\pp\in \Spec(R)$ such that the characteristic of $k(\pp)$ is coprime to $p$. In this case, note that $\P(H,\aa,\pp)=\P(H',\aa,\pp)$ for any subgroup $H\leq G$ by definition. This explains the relevance of the characteristic of $k(\pp)$ in the previous result. 
\end{remark}

\begin{corollary}\label{cor:inclusions of P}
Let $H$ and $H'$ be subgroups of $G$, and let $\pp\in \Spec(R)$ such that the characteristic of $k(\pp)$ is $p$.  If $\P_{G; R}(H, \aa,\pp) \subseteq \P_{G; R}(H', \aa',\pp)$ then $H' \leq_G H$. Therefore if $\P_{G; R}(H, \aa,\pp) = \P_{G; R}(H', \aa',\pp)$ then $H$ and $H'$ are $G$--conjugate.
\end{corollary}

\begin{proof}
    This follows from Lemma \ref{lem:koszul_conjugate_subgroups_ker} since 
    \[\widetilde{\kos}_G(H; R) \in \P_{G; R}(H, \aa,\pp) \subseteq \P_{G; R}(H', \aa',\pp)\]
    holds if and only if $H' \leq_G H$. The second part follows by symmetry of the argument.
\end{proof}

We are almost ready to describe the set $\Spc(\K(G;R))$ in terms of $\Spc(\K(G;k(\pp)))$. We need some preparation.

\begin{recollection}\label{primes under homeo induced by conjugation}
   Let $G$ be a subgroup of $G'$, and let $g$ be an element in $G'$. Then the conjugation homomorphism $c_g\colon G \xrightarrow{} G^g, \, h\mapsto h^g$,  induces a tt-equivalence $c_g^*\colon \K(G^g; R) \xrightarrow{\sim} \K(G; R)$ and hence a homeomorphism
\[
(\_)^g \coloneqq \Spc(c_g^*)\colon \Spc(\K(G;R)) \xrightarrow{\sim} \Spc(\K(G^g;R)); \,  \P \mapsto \P^g.
\]
In particular, if $g \in G$, the functor $c_g^*\colon \K(G;R) \xrightarrow{\sim} \K(G;R)$
 is isomorphic to the identity and therefore $\P^g = \P$ for all   $\P \in \Spc(\K(G;R))$. See Section 4.1 of \cite{BG23b} for more details.
\end{recollection}

\begin{recollection}\label{primes P under restriction}
We want to record the behavior of the primes $\P(H,\aa,\pp)$ under tt-functors induced by group homomorphisms. This follows a similar pattern as \cite[Remark 7.6]{BG23b}. 
    Let $\alpha\colon G \to G'$ be a group homomorphism, and let $H$ be a subgroup of $G$. Let $H' = \alpha(H) \leq G'$. Then $\alpha(N_G(H)) \leq N_{G'}(H')$ and hence induces a morphism $\bar{\alpha}\colon \WGH \to \Weyl{G'}{H'}$. Let $\pp\in \mathrm{Spec}(R)$. We claim that
 \[
 \Spc(\alpha^*)(\P_G(H, \aa,\pp)) = \P_{G'}(H', \Spc(\bar{\alpha}^*) (\aa),\pp)
 \]
     in $\Spc(\K(G'; R))$, where $\alpha^*$ is the restriction functor along $\alpha$; similarly the functor  $\bar{\alpha}^*\colon \mathbf{D}_\mathrm{perm}(\WGH; k(\pp)) \to \mathbf{D}_\mathrm{perm}(\Weyl{G'}{H'}; k(\pp))$ is the derived restriction functor induced by $\bar{\alpha}$. Indeed, 
     we only need to verify that the following square is commutative.
     \begin{center}
         \begin{tikzcd}[column sep = 3em]
             V_{\Weyl{G}{H};k(\pp)} \arrow[r,"\Spc(\bar \alpha^\ast)"] \arrow[d,"\Spc(\Upsilon)"'] & V_{\Weyl{G'}{H'};k(\pp)}  \arrow[d,"\Spc(\Upsilon)"]  \\        
            \Spc( \K({\Weyl{G}{H}};k(\pp)) ) \arrow[r,"\Spc( \alpha^\ast)"] \arrow[d,"\Spc(\iota_\pp^\ast)"'] & \Spc(\K({\Weyl{G'}{H'}};k(\pp)))  \arrow[d,"\Spc(\iota_\pp^\ast)"]  \\ 
             \Spc(\K(G;R)) \arrow[r,"\Spc(\alpha^\ast)"] & \Spc(\K(G';R))
         \end{tikzcd}
     \end{center}
    The commutativity of the diagram follows since the functors $\Upsilon^?$ are well behaved under group homomorphisms (see \cite[Remark 4.1]{BG23b}), the same holds for the base change functors $\iota_\pp^\ast$, and the category $\K(-;R)$ is functorial along group homomorphisms (see Remark \ref{functoriality of K}).
    In particular, we get the following two cases. 
    \begin{enumerate}
        \item Let $H\leq K\leq G$, and $\aa\in V_{\Weyl{K}{H};k(\pp)}$. Then 
        \[ \rho_H(\P_K(H,\aa,\pp))=\P_G(H,\bar\rho_H(\aa),\pp).\]
        \item Let $G\leq G'$, $g\in G'$, and $\aa\in V_{\WGH;k(\pp)}$. Then 
        \[
         \P_G(H,\aa,\pp)^g= \P_{G^g}(H^g,\aa^g,\pp).
        \]
    \end{enumerate}
\end{recollection}

\begin{recollection}\label{image of restriction}
    Let $H\leq G$, and let $(A_H,\mu,\eta)$ be the commutative algebra with $A_H=R(G/H)$, multiplication given by $\mu(\gamma\otimes \gamma')=\gamma $ if $\gamma=\gamma'$ in $G/H$ or $0$ otherwise,  and  unit $\eta(1)=\sum_{\gamma\in G/H} \gamma$.  This algebra is separable and has finite degree; a section is given by the map $\gamma \mapsto \gamma\otimes \gamma$. Moreover, following \cite{Bal17}, we can identify modules in $\K(G;R)$ over $A_H$ with the category $\K(H;R)$ and via this identification, restriction becomes extension of scalars and induction becomes the forgetful functor.  
   Then we can apply \cite[Theorem 3.19]{Bal16} to obtain a coequalizer of topological spaces 
   \[\bigsqcup_{[g]\in H\backslash G/H}\mathrm{Spc}(\K(H\cap{}^gH;R))\rightrightarrows \mathrm{Spc}(\K(H;R))\xrightarrow{\rho^G_H} \mathrm{supp}(A_H)\] 
   where one of the left hand side arrows is induced by restriction and the other is induced by conjugation along $[g]$ followed by restriction. Further details appear in \cite[Proposition 4.7]{BG23b}.
\end{recollection}

\begin{proposition}\label{psi check is injective}
    Let $H $ be a subgroup of $G$, and $\pp\in \Spec(R)$. If $\P_{G; R}(H, \aa,\pp) = \P_{G; R}(H, \aa',\pp)$, then $\aa = \aa'$. In other words, the map 
    \[
    \check\psi^{G,H}_\pp\colon V_{\WGH;k(\pp)}\to \Spc (\K(G;R))
    \]
    is injective.
\end{proposition}

\begin{proof}
Note that the statement is trivial if the characteristic of $k(\pp)$ is not $p$ since $V_{\WGH;k(\pp)}$ is a point. Hence, assume that $k\coloneqq k(\pp)$ has characteristic $p$.  The case $H=G$ is trivial since $V_{1;k}$ is a point. The case $H=1$ follows by noticing that $\check{\psi}^1_\pp=   \Spc(\Upsilon^G)\circ \Spc(\lambda_\pp)$ and each constituent is injective; indeed, the first map is injective since $\Upsilon^G$ is a localization, and $\Spc(\lambda_\pp)$ is injective by Remark \ref{cohomological primes from residue fields}. 
Now, assume that $H\leq G$ is a proper subgroup. First, assume that $H$ is normal in $G$. Note that the composite 
\[
\perm(G/H;R)\xrightarrow[]{\Infl} \perm(G;R) \xrightarrow[]{\Psi^H_\pp}\perm(G/H;k) 
\]
identifies with base change along $R\to k$. This observation together with compatibility of the localizations $\Upsilon$ with base change, gives us the following commutative diagram.
\[
\begin{tikzcd}
    V_{G/H;k} \ar[rrd,"\check\psi^{G/H,1}_\pp"'] \ar[r,"\Spc(\Upsilon^{G/H})"] &  \Spc(\K(G/H;k)) \ar[r,"\psi^{G,H}_\pp"] \ar[rd,"\Spc(\iota_\pp^\ast)"]  & \Spc(\K(G;R)) \ar[d,"\Spc(\mathrm{Infl_{G/H}^H})"] \\
    & & \Spc(\K(G/H;R))
\end{tikzcd}
\]
The previous argument gives us that $\check\psi^{G/H,1}_\pp$ is injective, hence so is $\check{\psi}^{G,H}_\pp$ which corresponds to the composite of the top arrows in the above diagram. This completes the normal case. 

For the general case, the argument is essentially the same as the one in the proof of \cite[Proposition 4.14]{BG23b}. For completeness, let us include the full argument. Set $N=N_G(H)$. By definition, we have 
\[\rho^G_N(\P_{N; R}(H, \aa,\pp)) = \rho^G_N(\P_{N; R}(H, \aa',\pp)).\] 
 Now, by Recollection \ref{image of restriction} we obtain that there is $g\in G$ together with a point $\mathcal{Q}$ in $\mathrm{Spc}(\K(N\cap{}^gN;R))$ such that 
 \begin{equation}\label{image of Q}
     \rho^N_{N\cap{}^gN}(\mathcal{Q})= \P_{N; R}(H, \aa,\pp) \mbox{ and }  {}^g(\rho^{{}^gN}_{N\cap{}^gN}(\mathcal{Q}))= \P_{N; R}(H, \aa',\pp)
 \end{equation}
Moreover, $\mathcal{Q}$ is of the form  $\P_{N\cap{}^gN;R}(L, \aa'',\pp)$ for some $L\leq N\cap{}^gN$, and some cohomological prime $\aa''\in V_{N\cap{}^gN /\!\!/ L; k(\pp)}$ (See Remark \ref{all points in K(G) are come from check psi}). Now, using Recollection \ref{primes P under restriction}, we know the behavior of the $\mathcal{Q}$ under the map on Balmer spectra induced by restriction:  
\[
\P_{N; R}(H, \aa,\pp)=\rho^N_{N\cap{}^gN}(\mathcal{Q})=\P_{N;R}(L,\bb,\pp) 
\]
and similarly for the map induced by conjugation 
\[
\P_{N; R}(H, \aa',\pp)= {}^g(\rho^{{}^gN}_{N\cap{}^gN}(\mathcal{Q}))= \P_{N;R}(L^g,\bb',\pp)
\]
for suitable cohomological primes $\bb\in V_{N/\!\!/ L;k(\pp)}$ and $\bb'\in V_{N/\!\!/ L^g;k(\pp)}$. By Corollary \ref{cor:inclusions of P}, we deduce that $L\sim_N H\sim_N L^g$. But $H$ is normal in $N$, and hence $L=H=L^g$. In particular, $g\in N$. Thus $N\cap N^g=N$ and $N^g=N$. These equalities applied to Equation \ref{image of Q} give us that $\P_{N;R}(H,\aa,\pp)=\mathcal{Q}=\mathcal{Q}^g=\P_{N,R}(H,\aa',\pp)$. Since $H\trianglelefteq N$, we can invoke the normal case which tells us that 
\[
\check \psi^{N;H}_p\colon V_{N/H;k}\to \mathrm{Spc}(\K(N;R))
\]
is injective. We conclude that $\aa=\aa'$ as we wanted.  
\end{proof}

\begin{theorem}\label{primes in spc up to conj}
    Every prime ideal of $\K(G;R)$ is of the form $\P_{G; R}(H, \aa,\pp)$ for some subgroup $H$ of $G$, some $\pp\in \Spec(R)$ and some cohomological prime $\aa\in V_{\WGH;k(\pp)}$. Moreover, the following properties hold.
    \begin{enumerate}
        \item If the characteristic of $k(\pp)$ is $p$, then $\P_{G; R}(H, \aa,\pp) = \P_{G; R}(H', \aa',\pp)$  if and only if there exists  $g \in G$ such that $H' = H^g$ and $\aa = \aa^g$. 
        \item If the characteristic of $k(\pp)$ is not $p$, then $\P_{G; R}(H, \aa,\pp) = \P_{G; R}(H', \aa',\pp)$ for any $H\leq G$, and any $\aa\in V_{\WGH;k(\pp)}$. 
        \item $\P_{G; R}(H, \aa,\pp) \not= \P_{G; R}(H', \aa',\pp')$ as soon as $\pp\not= \pp'$.
    \end{enumerate} 
\end{theorem}

\begin{proof}
    The first claim follows from Remark \ref{all points in K(G) are come from check psi}. Part $(1)$ is an immediate consequence of Corollary \ref{cor:inclusions of P}, Proposition \ref{psi check is injective}, and the behavior of primes under restriction along homomorphisms of groups (see Recollection \ref{primes P under restriction}). Part $(2)$ is Remark \ref{rem:primes in coprime char}. Part $(3)$ is Corollary \ref{cohomological primes over different fibers}. 
\end{proof}

We can now combine the previous results to give an alternative set-theoretic description of the Balmer spectrum of $\K(G;R)$.

\begin{proposition}\label{Spc as a set}
   There is a set bijection induced by base change along all residue fields of $R$: 
   \[
   \bigsqcup_{\pp\in\mathrm{Spec}(R)} \mathrm{Spc}(\K(G;k(\pp))) \xrightarrow[]{} \mathrm{Spc}(\K(G;R)).
   \]
 \end{proposition}

\begin{proof}
First of all, consider the following diagram. 
\begin{center}
        \begin{tikzcd}[column sep = 4em]
            \displaystyle\bigsqcup_{\pp\in \Spec(R)} \Spc(\K(G;k(\pp))) \arrow[d, "\mathrm{Comp}"'] \arrow[r, "\Spc(\iota_\pp^\ast)"] & \Spc (\K(G;R)) \arrow[d,"\mathrm{Comp}"]\\
           \displaystyle\bigsqcup_{\pp\in \Spec(R)} \Spec(k(\pp)) \arrow[r,"\sim"] & \mathrm{Spec}(R)
        \end{tikzcd}
    \end{center} 
and note that this diagram is commutative by the naturality of the comparison map. Hence, we only need to prove that the map $\Spc(\iota_\pp^\ast)$ is injective for each $\pp\in \Spec(R)$. If the characteristic of $k(\pp)$ is coprime to $p$, then this is trivial. Assume otherwise. The result follows from  Theorem \ref{primes in spc up to conj}. Indeed, this theorem tells us that the  map 
\[
\displaystyle\bigsqcup_{H\in \mathrm{Sub}(G)/G} V_{\WGH;k(\pp)} \xrightarrow[]{\bigsqcup\Spc(\iota_\pp^\ast)\circ \check\psi^{G,H}_\pp} \Spc(\K(G;R)) 
\]
is injective. But this map, by definition, factors through $\Spc(\K(G;k(\pp)))$. Then the result follows by \cite[Proposition 7.32]{BG23b} which states that $\bigsqcup \psi^{G,H}_\pp$ is a bijection. 
\end{proof}

\begin{remark}
    The results in this section recover and refine \cite[Proposition 6.13]{Gom25} which also establish the previous set-theoretic bijection by different methods.  
\end{remark}


\section{Reduction to the modular locus and to elementary abelian groups}


Let $G$ be a $p$-group and $R$ a commutative Noetherian ring. In this section, we reduce the computation of the Balmer spectrum of $\K(G;R)$ to understanding the modular locus, together with a partial reduction of the latter to the elementary abelian case. 

\begin{theorem}
    \label{thm-closure-R[1/n]}
    The closure of the ordinary locus of $\K(G;R)$ in $\Spc(\K(G;R))$ is contained in the cohomological open $V_{G;R}$. 
\end{theorem}

\begin{proof}
      Fix a point $\P(\qq)=\P(H',\aa',\qq)$ in the ordinary locus. Note that this is completely determined by $\qq$. Indeed, it corresponds to the kernel of tensoring with $k(\qq)$. It is enough to show that $\P(\qq)$ cannot specialize to a point 
      \[
      \P_{G;R}(H,\aa,\pp)\in \Spc(\K(G;R))\backslash V_{G;R}.
      \]
      In view of Theorem \ref{primes in spc up to conj},  we deduce that such a prime satisfies that $H$ must be non-trivial and $\pp$ is such that $\mathrm{char}( k(\pp))=p$. We claim that any prime $\Spc(\K(G;R))\backslash V_{G;R}$ must contain a permutation module of the form $R(G/K)$ for some non-trivial $K\leq G$. Indeed, by \cite{BG23b} the closed points of $\K(G;k(\pp'))$ have the form the kernel of $\Upsilon \circ \Psi^K$ for $K\leq G$. Therefore, $k(\pp')(G/L)$ belongs to such kernel for any $L$, $K\not\leq_G L$. Then a prime $\P(H,\aa,\pp)\in \Spc(\K(G;R))\backslash V_{G;R}$ must contain $R(G/K)$ for some  $K$ that contains no conjugate of $H$. On the other hand, note that $R(G/H)\otimes_R k(\qq)\simeq k(\qq)(G/H)\not=0$, thus $R(G/H)\notin \P(\qq)$. Hence $\P(H,\aa,\pp)\not\subseteq \P(\qq)$, as we wanted.  
\end{proof}

\begin{remark}
    \label{rem-erroneous-construction}
     In an earlier version of this document, we relied on an erroneous construction, namely, we claimed the existence of a continuous map between Balmer spectra over an arbitrary commutative Noetherian base which generalize the behavior of the map induced by modular fixed points when restricted to fields or rings of prime characteristic. While such set-theoretic map exists, it cannot be continuous as a consequence of the previous theorem. Indeed, continuity will create specializations between the ordinary locus and the complement of the cohomological open.  
\end{remark}

\begin{lemma}
\label{lemma-reduction-modp}
Assume that $a\in R$ is not a zero-divisor, and consider the geometric functor induced by base change along $R\to R/aR\eqqcolon S$
\[
f^\ast_a\colon \T(G;R) \to \T(G;S).
\]
Then its right adjoint $(f_a)_\ast$ preserves compact objects. 
\end{lemma}

\begin{proof}
It is enough to verify that $(f_a)_\ast(S(G/H))$ is compact for any $H\leq G$. Now, let $C(a)$ be the cone of $R\xrightarrow[]{a}R$. Note that $f_a^\ast C(a)\simeq S\oplus S[1]$. Hence there is a map $\alpha \colon C(a)\to (f_a)_\ast(S)$. 
We claim that $\alpha$ is an isomorphism. Indeed, the same argument as in the proof of \cite[Proposition 4.7]{Gom25} gives us that $(f_a)_\ast(S)$ lies in $\mathbf{D}(R)\subseteq \T(G;R)$. Then it is enough to show that $\alpha$ is a quasi-isomorphism. Since $a$ is not a zero-divisor, we have a short exact sequence 
\[
0\to R\xrightarrow[]{a} R \to S \to 0
\]
This determines a triangle in $\mathbf{D}(R)$. It follows that $\alpha$ is a quasi-isomorphism, as desired, and hence $(f_a)_\ast(S)$ is compact.  For the general case, the projection formula gives us the following isomorphism.
\[
(f_a)_\ast(S(G/H))\simeq (f_a)_{\ast}(f_a^\ast (R(G/H)) \otimes_S S)\simeq R(G/H)\otimes_R (f_\ast)(S). 
\]
This tensor product is a compact object since their constituents are compact. This proves the claim.   
\end{proof}

\begin{theorem}
    \label{thm-modular-locus}
    Let $R$ be a commutative Noetherian ring such that $p$ is not a zero-divisor. Then the map 
    \[
    \Spc(f_p^\ast)\colon \Spc(\K(G;R/pR)) \to \Spc(\K(G;R))
    \]
    is a closed immersion with image the modular locus of $\K(G;R)$.
\end{theorem}

\begin{proof}
    By \cite[Remark 13.26]{barthel2023cosupport} combined with Lemma \ref{lemma-reduction-modp}, we obtain that $\Spc(f_p^\ast)$ is a closed map with image the support of $(f_p)_\ast(R/pR)\simeq \mathrm{cone}(R\xrightarrow[]{p}R)$ which is precisely the modular locus. The fact that it's injective follows from Proposition \ref{Spc as a set} and the fact that primes whose residue fields have characteristic $p$ lie in $\Spec(R/pR)\subseteq \Spec(R)$. 
\end{proof}

\begin{remark}
    The previous results give a complete description of the Balmer spectrum of $\K(G;\mathbb Z)$ in terms of $\K(G;\mathbb F_p)$ and $\Spec^h(H^\bullet(G;\mathbb Z))$. A further discussion appears in Section \ref{sec: examples}.  
\end{remark}

We now reduce the understanding of the Balmer spectrum of $\K(G;R)$ to the elementary abelian case, at least for rings of characteristic $p$. This follows closely \cite{BG23b}. We begin by recalling the construction of the \textit{category of elementary abelian sections $\mathcal{E}(G)$ of $G$}; see \cite[Construction 11.1]{BG23b}.

\begin{recollection}\label{rec category of p sections}
     Let $\mathcal{E}(G)$ denote the category whose objects are pairs $(H,K)$ such that $K\unlhd H\leq G$, and $H/K$ is elementary abelian; a morphism in this category from  $(H,K)$ to $(H',K')$ corresponds to an element $g$ of $G$ such that 
     \[
     K'\leq K^g\leq H^g \leq H'
     \]
     and the composition of morphisms is given by the multiplication in $G$. 
\end{recollection}

\begin{construction}
    Let $R$ be a commutative Noetherian ring of characteristic $p$. For an object $(H,K)$ in $\mathcal{E}(G)$, we will consider the tt-category $\K(H/K;R)$. For a morphism $g$ in $\mathcal{E}(G)$ we define a tt-functor
    \[
    \K(g, R)\colon \K(H'/K';R)\xrightarrow[]{\Psi^{\bar K}} \K(H'/K^g;R) \xrightarrow[]{\mathrm{Res}} \K(H^g/K^g;R)\xrightarrow[]{c_g^\ast}\K(H/K;R)
    \]
 where $\Bar{K}=K^g/K'$; we are using that $(H'/K')/\Bar{K}=H'/K^g$, and that $H^g/K^g\leq H'/K^g$.
 This defines a functor from $\mathcal{E}(G)^\textrm{op}$ to tt-categories, and composing with $\Spc(-)$ gives us a functor from $\mathcal{E}(G)$ to topological spaces.   
 
 On the other hand,  for each object $(H,K)$ of $\mathcal{E}(G)$, we have a commutative square of  tt-functors
 \[
 \begin{tikzcd}
     K(G;R) \ar[d,"\Res^G_H"'] \ar[r,"\Psi^{K}"] & \K(\WGK;R) \ar[d,"\Res^{\WGK}_{H/K}"]  \\ 
       \K(H;R)  \ar[r,"\Psi^K"] &  \K(H/K;R)
 \end{tikzcd}
 \]
using that $H\leq N_G(K)$. We write $\varphi_{(H,K)}$ to denote the induced map on Balmer spectra by the diagonal of the above square. In particular, we obtain a continuous map 
 \begin{equation}\label{comparison with colim}
\varphi\colon \underset{(H,K)\in \mathcal{E}(G)}{\colim} \mathrm{Spc}(\K(H/K;R))\to \mathrm{Spc}(\K(G;R))     
 \end{equation}
whose component at $(H,K)$ is precisely $\varphi_{(H,K)}$. 
\end{construction}

\begin{theorem}\label{colimit theorem}
 Let $G$ be a $p$-group, and $R$ be a commutative Noetherian ring of characteristic $p$.  The comparison map \[\varphi\colon \underset{{(H,K)\in \mathcal{E}(G)}}{\colim} \mathrm{Spc}(\K(H/K;R))\to \mathrm{Spc}(\K(G;R))  \] from Equation \ref{comparison with colim} is a homeomorphism. 
\end{theorem}

\begin{proof}
  First, note that each $\varphi_{(H,K)}$ is closed. Indeed, it is the composition of two closed maps. First, for a normal subgroup $N\unlhd G$, we have that $\psi^N$ is closed since it has a continuous retraction given by $\Spc(\mathrm{Infl}_{G/N}^G)$ and its image is the support of the object $\bigotimes_{N\not\leq L} \mathrm{kos}_G(L;R)$; see \cite[Proposition 7.18]{BG23b}. For an arbitrary subgroup $L\leq G$, we have that $\rho^G_L=\Spc(\Res^G_L)$ is closed by \cite[Proposition 4.7]{BG23b}. Note that the latter holds over an arbitrary commutative base. It follows that $\rho$ is closed as well. 
  
  We are left to verify that $\rho$ is a bijection. To this end,  set $\pp\in \mathrm{Spec}(R)$. Then the comparison map 
 \[\varphi_{k(\pp)}\colon \underset{{(H,K)\in \mathcal{E}(G)}}{\colim} \mathrm{Spc}(\K(H/K;k(\pp)))\to \mathrm{Spc}(\K(G;k(\pp)))  \]
 is a homeomorphism by \cite[Theorem 11.10]{BG23b}. Now, let us consider the following commutative diagram obtained from base change along all residue fields of $R$. 
 \begin{center}
    \begin{tikzcd}
       \displaystyle \bigsqcup_{\pp\in\mathrm{Spec}(R)} 
 \underset{(H,K)\in \mathcal{E}(G)}{\colim} \Spc(\K(H/K;k(\pp))) \arrow[r,"\sqcup\varphi_{k(\pp)}"] \arrow[d] & \displaystyle \bigsqcup_{\pp\in\mathrm{Spec}(R)} \mathrm{Spc}(\K(G;k(\pp))) \arrow[d] \\ 
\displaystyle  \underset{(H,K)\in\mathcal{E}(G)}{\colim} \Spc(\K(H/K;R)) \arrow[r,"\varphi"] & \mathrm{Spc}(\K(G;R))
    \end{tikzcd}
\end{center}
    here we are using that coproducts commute with colimits to rewrite the top left corner. Now, note that the top map is a bijection by the previous observation; the  vertical maps are bijections by Proposition \ref{Spc as a set}. It follows that $\varphi$ is also a bijection as we wanted.
\end{proof}

\begin{remark}
    In a nutshell, the results of this section allow us to reduce the understanding of the topology of $\Spc(\K(G;R))$, for commutative Noetherian rings that do not contain $p$ as a zero-divisor and $p$-groups,  to the understanding of its cohomological open and its modular locus, and a further reduction of the latter to the understanding of the category of $p$-sections together with the understanding of the topology of $\Spc(\K(E,R/pR))$ for elementary abelian $p$-groups. 
    
    A similar reduction can be made for general groups. Namely, for a finite group $G$, we can use the same argument as in the proof of \cite[Lemma 6.1]{Gom25} to reduce the understanding of the topology of $\Spc(\K(G;R))$ to understanding of the orbit category of $G$ with isotropy on $p$-power order subgroups, together with the understanding of the topology of $\Spc(\K(S;R))$ for $p$-groups (here we let $p$ vary). Altogether, this gives us that the topology of $\Spc(\K(G;R))$ for an arbitrary finite group $G$ can be understood from the topology of $\Spc(\K(E;R/pR))$ for elementary abelian groups, together with the cohomological open and information about the category of $p$-sections and the orbit category with isotropy on $p$-groups, for all primes dividing the order of $G$. 
\end{remark}

\section{Tensor invertible objects}
\label{section:invertible objects}

Throughout this section, we let $G$ be a $p$-group. Unfortunately, we will soon need to place some additional assumptions on our ground ring, but for the moment, let us consider an arbitrary commutative Noetherian ring $R$.

Our goal in this section is to introduce some tensor-invertible objects in the category $\K(G;R)$ that will allow us to extend Balmer-Gallauer twisted cohomology ring in greater generality. Invertible objects in this category have been extensively studied in the modular case by Miller \cite{Mil24}. More recently, over a general Noetherian base in \cite{GM26a} by the second author joint with Miller. 

\subsection*{Frobenius structure on lattices.} Recall that the group algebra $kG$ is  a Frobenius algebra when the ground ring $k$ is a field and hence, projective $kG$--modules and  injective $kG$--modules agree. However, the group algebra is not necessarily Frobenius over arbitrary commutative rings, and in general, projectives do not have to agree with injectives. We will restrict ourselves to work in the full subcategory $\mathrm{Mod}(G;R)$  of $\mathrm{Mod}(kG)$ on $RG$--lattices, that is, $RG$--modules that are projective as $R$--modules. Once and for all, let us fix the exact structure on  $\mathrm{Mod}(G;R)$ induced by the split exact structure on $\mathrm{Proj}(R)$. In fact, this exact structure turns $\mathrm{Mod}(G;R)$ into an exact Frobenius category.  We will record the following result for ease of reference; a proof can be found in \cite[Section 2]{BBIKP}. 

\begin{proposition}\label{Frobenius structure}
The category of $RG$--lattices $\Mod(G;R)$ is an exact subcategory of $\Mod(RG)$ with the exact structure given by the split exact sequences of $R$--modules. In fact, $\Mod(G;R)$ is a Frobenius exact category where the injectives and projectives are given as the additive closure of the projective $RG$--modules. 
\end{proposition}

\begin{remark}
Note that permutation modules are indeed lattices. We consider   $\mathrm{Perm}(G;R)$ as a full subcategory of $\mathrm{Mod}(G;R)$.    
\end{remark}

\subsection*{Invertible objects for $C_p$} Let $R$ be a commutative Noetherian ring,  $C_p=\langle \sigma\rangle$ be the cyclic group of prime order $p$, and $RC_p$ be the group algebra. Note that $RC_p = R[\sigma]/(\sigma^p - 1)$. Recall that there is a periodic resolution which computes the cohomology of $C_p$: 
\begin{align}\label{eq:standar resolution}
    \ldots \xrightarrow{\mathbf{N}} RC_p  \xrightarrow{\sigma-1}  RC_p\xrightarrow{\mathbf{N}} RC_p \xrightarrow{\sigma -1} RC_p \xrightarrow{\epsilon} R\to 0
\end{align}

\noindent where $\mathbf{N}$ denotes the \textit{norm map} which is given by multiplying with $1+\sigma+\ldots+ \sigma^{p-1}$, and $\epsilon$ denotes the augmentation map.

\begin{definition}\label{def: uN}
  Let $u_{p,R}$ denote the complex of $RC_p$--modules: 
  \begin{enumerate}
      \item $0 \to  RC_p \xrightarrow{\epsilon} R \to 0$ if $p=2$; or
      \item $0 \to RC_p \xrightarrow{\sigma-1} RC_p \xrightarrow{\epsilon} R \to 0$, if $p>2$. 
  \end{enumerate}
In both cases, the copy of $R$ is considered in degree 0.  Moreover, note that the complex $u_{p,R}$ is obtained as a brutal truncation of the complex given in \ref{eq:standar resolution}. If the context is clear, we will drop $R$ from the notation and simply write $u_p$. 
\end{definition}

\begin{remark}
Recall that the tensor structure on $\mathrm{Perm}(G;R)$ given by the tensor product with diagonal $G$--action endows $\K(G;R)$ with a symmetric monoidal structure where the monoidal unit is $R$ viewed as a complex concentrated in degree 0. In particular, makes sense to talk about \textit{invertible objects}; explicitly, we say that $x\in \K(G;R)$ is $\otimes$-invertible if there exists $y \in \K(G;R)$ such that $x\otimes y\simeq R$. We will see that  $u_{p,R}$ are examples of invertible objects extending the observation made in \cite{BG23b} when $R$ is a field.     
\end{remark}

\begin{lemma}\label{u_p is invertible}
    The object $u_{p, R}$ is a $\otimes$-invertible object in $\K(C_p;R)$. Moreover, for $\pp\in \Spec(R)$, the base change functor along the residue field $R\to k(\pp)$ give us an isomorphism
    \[
    \iota_\pp^\ast( u_{p, R} ) \simeq u_{p, k(\pp)}.
    \] 
\end{lemma}

We need some preparations first.

\begin{proposition}\label{homology of X otimes its dual}
    Let $X$ and $Y$ be bounded complexes of finitely generated $RG$--modules satisfying the following properties:
    \begin{enumerate}
        \item For any $n$, both $X_n$ and $Y_n$ are free as $R$--modules. 
        \item $X$ has homology in a single degree $i$ and it is projective as $R$--module and $Y$ has homology in a single degree $-i$ and it is projective as $R$--module.
    \end{enumerate}
    Then $X\otimes Y$ has homology concentrated in degree 0.
\end{proposition}

\begin{proof}
    By $(1)$, we obtain that any term of $X\otimes Y$ is free as $R$--module. Now, by the Künneth formula, we have that the homology of $X\otimes Y$ in degree $n$ is given by  
    \[
    0\to \bigoplus_{r+s=n}H_r(X)\otimes H_s(Y)\to H_n(X\otimes Y)\to \bigoplus_{s+r=n-1}\mathrm{Tor}^1(H_r(X),H_s(Y))\to 0.
    \]
    By $(2)$, we deduce that $\mathrm{Tor}^1$ vanishes, and the homology must be concentrated in degree 0 as we wanted. 
\end{proof}

\begin{proof}[Proof of Lemma \ref{u_p is invertible}]
Note that $u_p$ satisfies the hypothesis of Proposition \ref{homology of X otimes its dual}, hence we obtain that the homology of $X= u_p\otimes u_p^\ast$ is concentrated in degree 0. Here $u_p^\ast$ denotes the dual complex of $u_p$. Moreover, since the non-trivial homology group of $u_p$ is $R$, we deduce that $H_0(X)\cong R$. 

We will first deal with $p$ odd. In this case, $X_n$  is non-trivial for $n\in \{-2,\ldots, 2\}$, and it is projective as $RC_p$--module as long as $n\not=0$. We claim that the differentials $d_2\colon X_2\to X_1$ and $d_{-1}\colon X_{-1}\to X_{-2}$ must be split injective and split surjective, respectively. Indeed, since projective $RC_p$--modules are the same as injective $RC_p$--modules with respect to the exact structure of $\mathrm{Mod}(C_p;R)$, it is enough to verify that $d_2$ and ${d_{-1}}$ are admissible mono and admissible epi, respectively.  But this follows since  $X$ has homology concentrated in degree 0 and it is free as $R$--module.
As a consequence, we can `remove' the extremes of the complex $X$. That is, we can write $X$ as  $X'\oplus Y \oplus Z$ with
\[
Y= 0\to X_2\xrightarrow{\simeq} X_2\to 0, \quad \mbox{and} \quad Z=0\to X_{-2}\xrightarrow{\simeq} X_{-2}\to0
\]
 where the first $X_2$ is in degree 2 and the first $X_{-2}$ is in degree -1, and some complex $X'$ with $X'_n=0$ except in degree -1,0 and 1. Since $Y$ and $Z$ are contractible, we obtain that $X$ is homotopy equivalent to $X'$. Since   $X'_{\pm 1}$ is a summand of $X_{\pm 1}$, both $X'_1$ and $X'_{-1}$ are projective $RC_p$--modules, so we can apply the same argument to `remove' the extremes. This leads us to a complex $X''$ concentrated in degree 0, and which is homotopy equivalent to $X$. Since the homology of $X$ in degree 0 is given by $R$ and trivial otherwise, we obtain that $X''$ must be $R$ as a complex concentrated in degree 0, as we wanted.  The case $p=2$ is completely analogous, except that we only need to apply the previous argument once. 

The second claim is an immediate consequence of the fact that the resolution given in \ref{eq:standar resolution} can be obtained from $R=\mathbb{Z}$ by base change. 
\end{proof}

\begin{remark}
 A different proof of Lemma \ref{u_p is invertible} is possible by invoking \cite[Theorem 4.8]{Gom25}. In this case,   we obtain  that the tt-functor 
 \[
 \prod \iota_\pp^\ast\colon\K(C_p; R) \to \prod_{\pp \in \mathrm{Spec}(R)} \K(C_p; k(\pp))
 \]
 is conservative. 
     Hence, it is enough to show that the objects $\iota_\pp^\ast(u_{p,R})\simeq u_{p, k(\pp)}$ are tensor invertible for each $\pp \in \mathrm{Spec}(R)$.
    If the characteristic of $k(\pp)$ is $p$ then this follows from \cite{BG23b}. Otherwise, note that $u_{p, k(\pp)}$ is isomorphic to $k(\pp)$ up to a shift depending on whether $p$ is even or odd. 
\end{remark}

\subsection*{Invertible objects in the general case} We extend the definition of the complexes $u_p$ to a more general context as follows.

\begin{definition}\label{def:u_N}
Let $G$ be a finite group and   $N$ be a normal subgroup of $G$ of  index $p$.  Then define 
\[
u_{N, R} \coloneqq \mathrm{Infl}_N^G (u_{p,R})
\]
where $\mathrm{Infl}_N^G\colon \K(G/N;R)\to \K(G;R)$ denotes the inflation functor obtained by the group homomorphism $G\to G/N$.  
\end{definition}

\begin{remark}
There are a few considerations to keep in mind regarding the previous definition:
\begin{enumerate}
\item The complex $u_{N,R}$ also depends on the choice of a generator of $G/N$ to make the identification with $C_p$. 
For now, observe that the isomorphism type of $u_{N,R}$ is independent of this choice (see \cite[Remark 12.5]{BG23b}).
\item The inflation functor is indeed a tt-functor; in particular, the objects $u_{N,R}$ are $\otimes$-invertible in $\K(G;R)$.
\item A \textit{concrete} description of the objects $u_{N,R}$ is as follows.   \[u_{N,R}= \left\{
        \begin{array}{ll}
           
            0\to  R(G/N) \xrightarrow{\epsilon} R \to 0 & \textrm{ if } p=2\\
            0 \to R(G/N)\xrightarrow{\sigma-1} R(G/N) \xrightarrow{\epsilon} R \to 0 & \textrm{ if } p>2.
        \end{array}
    \right. \]
\end{enumerate}
\end{remark}

\begin{remark}
Recall that tensoring any two invertible objects is again invertible. In particular, any tensor power of $u_{N,R}$ is again invertible in $\K(G;R)$. We will give a description of such tensor powers since they will play a crucial role in what follows. In fact, such a description generalizes the field case treated in  \cite[Lemma 12.6]{BG23b}.   
\end{remark}

\begin{notation}
    From now on, we will adopt Balmer-Gallauer's notation and write $2'$ to denote either $1$ for $p=2$, or $2$ for $p>2$.
\end{notation}

\begin{lemma}\label{tensor powers of uN}
Let $G$ be a finite group and  $N$ be a normal subgroup of $G$ of index $p$ and  $q \geq 1$. Then there is a canonical identification in $\K(G;R)$
    \[u_{N, R}^{\otimes q} \simeq (0 \to R(G/N) \xrightarrow{\sigma-1} R(G/N) \xrightarrow{\mathbf{N}} \cdots \xrightarrow{\sigma -1} R(G/N) \xrightarrow{\epsilon} R \to 0)\] 
    where the copy of $R$ in the right-hand side is considered in degree 0, and the left most $R(G/N)$ is in degree $q\cdot 2'$.
\end{lemma}

\begin{proof}
Let $Y$ denote the right hand side complex.  We will only deal with the case $p>2$, the other case follows in a similar pattern. First, assume that $G=C_p$. By Lemma \ref{u_p is invertible}, we know that $u_p$ is invertible, then $u_p^{\otimes q}$ is invertible for any $q\geq1$. In particular, it is well known that if $x\otimes y\simeq 1$ in a closed symmetric monoidal category, then we can identify $y$ with the dual of $x$, that is, $x^\ast\simeq y$, and for an invertible object $x$, there is a canonical isomorphisms from $x$ to its double dual $x^{\ast \ast}$. Fix $q\geq 2$. We will show that the dual $Y^\ast$ of $Y$ is an inverse for $X=u_p^{\otimes q}$, and this will imply the desired result. Note that $Y^\ast$ is given by\footnote{The complex $Y^\ast$ corresponds to the truncation of the standard resolution to compute the cohomology of $C_p$.} 
\[
0\to R \xrightarrow{\eta}  RG \xrightarrow{\sigma-1} RG \xrightarrow{\mathbf{N}} \cdots \xrightarrow{\sigma-1} RG \to 0
\]
with $R$ in degree 0, and $\eta$ the counit of the adjunction restriction-induction. We can use a slight variation of Proposition \ref{homology of X otimes its dual} to show that $X$ has homology concentrated in degree $2q$, while $Y^\ast$ has homology concentrated in degree $-2q$. In particular, the pair $X$ and $Y$ satisfies the hypothesis of  Proposition \ref{homology of X otimes its dual}, and we deduce that the homology of $X\otimes Y^\ast$ is concentrated in degree 0 and it is given by $R$. By inspection, $(X\otimes Y^\ast)_n$ is a projective a $RG$--module, for any $n\not=0$. Moreover, note that, as in the proof of Proposition  \ref{u_p is invertible}, the differentials $d_{2q}$ and $d_{-2q+1}$ are admissible mono and admissible epi in $\Mod(G;R)$, respectively. Hence,   we can `remove' the extremes of $X\otimes Y^\ast$ without changing the homotopy type. That is, we can write $X\otimes Y^\ast$ as $(X\otimes Y^\ast)'\oplus A \oplus B$ with
\begin{align*}
     A & = 0\to (X\otimes Y^\ast)_{2q}\xrightarrow{\simeq} (X\otimes Y^\ast)_{2q}\to 0, \mbox{ and} \\ 
      B & = 0\to (X\otimes Y^\ast)_{-2q}\xrightarrow{\simeq} (X\otimes Y^\ast)_{-2q}\to0
\end{align*}
where the first $(X\otimes Y^\ast)_{2q}$ is in degree 2q and the first $(X\otimes Y^\ast)_{-2q}$ is in degree $-2q+1$, and some complex $(X\otimes Y^\ast)'$ with $(X\otimes Y^\ast)'_n=0$ except for $n\in \{-2q+1,-2q+2,\ldots, 2q-1\}$. Since $A$ and $B$ are contractible, we have that $(X\otimes Y^\ast)$ is homotopy equivalent to $(X\otimes Y^\ast)'$. Again, by inspection, we have that $(X\otimes Y^\ast)'_n$ is a projective $RG$--module, for any  $n\not=0$. We apply the same argument until we get a complex concentrated in degree $0$ homotopy equivalent to $X\otimes Y^\ast$. Since we know the homology of $X\otimes Y^\ast$, we deduce that it must be homotopy equivalent to $R$ concentrated in degree $0$. This completes the case $G=C_p$. For the general case, we only need to inflate since inflation is a tensor triangulated functor. 
\end{proof}

\section{Twisted integral cohomology}\label{sec: twisted}

In this section, we define the twisted cohomology ring for any $p$-group and study certain properties that it has under base change along residue fields. Before jumping into the construction, let us place a restriction on our ground ring.

\begin{convention}\label{convention on R}
From now on, we will assume that whenever $p \neq 0$ in $R$, then $\mathrm{ann}_R(p) = 0$. In other words, if $p$ is not trivial in $R$, then we will assume that $p$ is not a zero-divisor.
\end{convention}

\begin{remark}
    We decided to place the previous restriction in order to simplify several of the arguments that will appear in the rest of this document, rather than for a conceptual reason.
\end{remark}

\begin{definition}
Let $N$ be a normal subgroup of $G$ of index $p$. We will define the following maps in $\K(G;R)$, which will depend on $p$ and $R$. Recall that the definition of the invertible elements $u_N$ also depends on this, see Definition \ref{def: uN}.  

\begin{enumerate}
    \item[$(C1)$]  If $p=2$, and $p$ is trivial in $R$, define $a_N\colon \mathbb{1}\to u_N$ and $b_N\colon \mathbb{1}\to u_N[-1]$ as follows: 
    \begin{center}
     \begin{tikzcd}
  \mathbb{1} \arrow[d, "a_N"']  = & 0 \arrow[r] \arrow[d] &  R \arrow[d,"1"] \\ 
 u_N = & R(G/N) \arrow[r,"\epsilon"] & R.
    \end{tikzcd} 
    \end{center}
    and 
    \begin{center}
     \begin{tikzcd}
  \mathbb{1} \arrow[d, "b_N"']  = & R \arrow[r] \arrow[d,"\eta"] &  0 \arrow[d] \\ 
 u_N[-1] = & R(G/N) \arrow[r,"\epsilon"] & R 
    \end{tikzcd}    
    \end{center}
 \item[$(C2)$] If $p=2$, and $p$ is non-trivial in $R$, define $a_N\colon \mathbb{1}\to u_N$ and $b_N\colon \mathbb{1}\to u_N^{\otimes 2}[-2]$ as follows: 
 \begin{center}
 \begin{tikzcd}
  \mathbb{1} \arrow[d, "a_N"']  = & 0 \arrow[r] \arrow[d] &  R \arrow[d,"1"] \\ 
 u_N = & R(G/N) \arrow[r,"\epsilon"] & R
    \end{tikzcd} 
    \end{center}
    and 
    \begin{center}
     \begin{tikzcd}
  \mathbb{1} \arrow[d, "b_N"']  = & R \arrow[r] \arrow[d,"\eta"] &  0 \arrow[d] \arrow[r] & 0 \arrow[d] \\ 
 u_N^{\otimes 2}[-2] = & R(G/N) \arrow[r,"\sigma -1"] & R(G/N)\arrow[r,"\epsilon"] & R.
    \end{tikzcd}
    \end{center}

    \item[$(C3)$] If $p>2$, and $p$ is trivial in $R$, define $a_N\colon \mathbb{1}\to u_N$, $b_N\colon \mathbb{1}\to u_N[-2]$ and $c_N\colon \mathbb{1}\to u_N[-1]$ as follows: 
    \begin{center}
    \begin{tikzcd}
  \mathbb{1} \arrow[d, "a_N"']  = & 0 \arrow[r] \arrow[d] & 0 \arrow[r] \arrow[d] &  R \arrow[d,"1"] \\ 
 u_N = & R(G/N)\arrow[r,"\sigma-1"] & R(G/N) \arrow[r,"\epsilon"] & R
    \end{tikzcd} 
    \end{center}
     and 
     \begin{center}
     \begin{tikzcd}
  \mathbb{1} \arrow[d, "b_N"']  = & R \arrow[r] \arrow[d,"\eta"] &  0 \arrow[d] \arrow[r] & 0 \arrow[d] \\ 
 u_N[-2] = & R(G/N) \arrow[r,"\sigma -1"] & R(G/N)\arrow[r,"\epsilon"] & R 
    \end{tikzcd}
    \end{center}
and
\begin{center}
 \begin{tikzcd}
  \mathbb{1} \arrow[d, "c_N"']  = & 0 \arrow[r] \arrow[d] &  R \arrow[d,"\eta"] \arrow[r] & 0 \arrow[d] \\ 
 u_N[-1] = & R(G/N) \arrow[r,"\sigma -1"] & R(G/N)\arrow[r,"\epsilon"] & R.
    \end{tikzcd}
\end{center}

\item[$(C4)$] If $p>2$, and $p$ is non-trivial in $R$, define $a_N\colon \mathbb{1}\to u_N$ and $b_N\colon \mathbb{1}\to u_N[-2]$ as follows:
\begin{center}
 \begin{tikzcd}
  \mathbb{1} \arrow[d, "a_N"']  = & 0 \arrow[r] \arrow[d] & 0 \arrow[r] \arrow[d] &  R \arrow[d,"1"] \\ 
 u_N = & R(G/N)\arrow[r,"\sigma-1"] & R(G/N) \arrow[r,"\epsilon"] & R
    \end{tikzcd}
    \end{center}
      and
      \begin{center}
     \begin{tikzcd}
  \mathbb{1} \arrow[d, "b_N"']  = & R \arrow[r] \arrow[d,"\eta"] &  0 \arrow[d] \arrow[r] & 0 \arrow[d] \\ 
 u_N[-2] = & R(G/N) \arrow[r,"\sigma -1"] & R(G/N)\arrow[r,"\epsilon"] & R. 
    \end{tikzcd}
 \end{center}
\end{enumerate}
\noindent In all cases, $\eta\colon R\to R(G/N)$ denotes the unit of the adjunction given by restriction and induction. If we need to emphasize the role of the ring, we write $a_{N,R}$, $b_{N,R}$, and $c_{N,R}$, respectively. 
\end{definition}


\begin{remark}
In case $(C2)$, note that there are no non-trivial maps from $\mathbb{1}$ to $u_N[-1]$. However, as it is well known, there is a non-trivial class in integral cohomology in degree 2 for the cyclic group of order 2. This justifies, in some sense, considering the maps $b_i\colon \mathbb{1} \to u_i^{\otimes 2}[-2]$.
\end{remark}

\begin{remark}
Note that the maps $c_N$ do not appear in case  $(C4)$. The reason is that if $p\not =0$ in $R$, then there is no non-trivial map $R\to R(G/N)$ which has image in the kernel of the augmentation map. 
\end{remark}

\begin{remark}\label{cN is nilpotent}
    Just as in the field case, the morphisms $c_{N,R}$ defined in the case $(C3)$ are tensor-nilpotent. 
    Indeed, given that $\{\iota_\pp^\ast\}_{\pp\in \Spec(R)}$ is jointly conservative on big categories, they detect tensor-nilpotent maps between compact objects. Hence, note that if $k(\pp)$ has characteristic different than $p$, then $\iota_\pp^\ast(c_{N,R})$ is nullhomotopic; if $k(\pp)$ has characteristic $p$, then $\iota(c_{N;R})$ identifies with $c_{N,k(\pp)}$ which is tensor-nilpotent by  
    \cite[Remark 12.8]{BG23b}. We conclude that $c_{N,R}^{\otimes2}$ is tensor-nilpotent. However, our proof does not show that $c_{N,R}$ squares to 0 as in the field case.  
\end{remark}

Let us record the behavior of the maps $a_N, \, b_N$ and $c_N$ under the restriction functor. But before, let us recall the meaning of the notation $2'$; we write $2'$ to denote either 2 in the cases $(C3)$ and $(C4)$ or 1 in the cases $(C1)$ and $(C2)$.   

\begin{remark}\label{invertibles under restriction}
Consider a subgroup $H$ of $G$. Note that $\mathrm{Res}^G_H(u_N)$ has two possibilities. If $H$ is contained in $N$, then the action of $H$ on $u_N$ is trivial, hence it is isomorphic to $\mathbb{1}[2']$. On the other hand, if  $H\not\leq N$, then $H\cap N$ is normal in $H$ and has index $p$. In this case, we can identify the complex $\mathrm{Res}^G_H(u_N)$ with $u_{H\cap N}$. This identifications allow us to describe the behavior of the maps $a_N$ and $b_N$ under the restriction functor. Explicitly, we have that 
\[
\mathrm{Res}^G_H(a_N)= \left\{
        \begin{array}{ll}
           
            0 & \textrm{ if } H\leq N\\
            a_{H\cap N} & \textrm{ if } H\not\leq N
        \end{array}
    \right.
\] and 
\[
\mathrm{Res}^G_H(b_N)= \left\{
        \begin{array}{ll}
           
            \mathrm{id}_\mathbb{1} & \textrm{ if } H\leq N\\
            b_{H\cap N} & \textrm{ if } H\not\leq N.
        \end{array}
    \right.
\]
Moreover, whenever $c_N$ is defined, we obtain 
\[
\mathrm{Res}^G_H(c_N)= \left\{
        \begin{array}{ll}
           
            0 & \textrm{ if } H\leq N\\
            c_{H\cap N} & \textrm{ if } H\not\leq N.
        \end{array}
    \right.
\]
\end{remark}

\begin{lemma}\label{aN, bN and CN are generators}
    Let $N_1,\ldots,N_l$ be normal subgroups of $G$ of index $p$, $q_1,\ldots,q_l$ be non-negative integers and $s$ be an arbitrary integer. Then any map 
    \[
    f\colon \mathbb{1}\to u_1^{\otimes q_1}\otimes \ldots\otimes u_l^{\otimes q_l}[s]
    \]
    is polynomial in $a_i$ and $b_i$ (and $c_i$ in case $(C3)$). Here we are using the subscript $i$ instead of $N_i$ to simplify notation. 
\end{lemma}

The proof given by Balmer and Gallauer for the field case extends to our generality without much difficulty. The relevant ingredients in their proof are the adjunction given by restriction-induction, the description of the tensor powers of the tensor-invertible objects $u_N$ (see \cite[Lemma 12.6]{BG23b}), and the behavior of the maps $a_N$ and $b_N$ under the restriction functor (see \cite[Proposition 12.10]{BG23b}). The analogous ingredients in our generality correspond to Lemma \ref{tensor powers of uN} and Remark \ref{invertibles under restriction}, and, of course, the adjunction holds in this case as well. We will include case $(C2)$ here for completeness; the other cases, as we already mentioned, follow a similar pattern.

\begin{proof}[Proof of Lemma \ref{aN, bN and CN are generators}]
 Let $p=2$, and assume that $2$ is not trivial in $R$, that is, we are considering the case $(C2)$.  
 Since there are no non-trivial maps from $\mathbb{1}\to \mathbb{1}[m]$ for any $m\not=0$, we deduce that the graded endomorphism ring of the monoidal unit $\mathbb{1}$ in $\K(G;R)$ is just $R$. 
 We are now ready to start with the proof. We will argue by induction on $l$ using the previous observation as the base for the induction. Suppose that $l>0$, and that all the subgroups $N_1,\ldots, N_l$ are different; otherwise, we could apply our induction hypothesis. We write $v$ to denote $u_1^{q_1}\otimes \ldots\otimes u_{l-1}^{q_{l-1}}[s]$, and $u$ to denote $u_l$ and $q$ to denote $q_l$. In particular, we can write $f$ as a morphism from $\mathbb{1}\to v\otimes u^{\otimes q}$.
 Now, we proceed by induction on $q\geq0$. The base follows by the first induction hypothesis, hence assume that $q>0$. We can use the description of the tensor powers of $u_N$ given in Lemma \ref{tensor powers of uN}, and consider the map $a_n=a_N\otimes \mathrm{Id}_{u^{\otimes (q-1)}}\colon u^{\otimes (q-1)}\to u^{\otimes q}$ which is the identity in all degrees different to $q$ and zero in degree $q$.

 Note that the cone of this map in $\K(G;R)$ is simply $R(G/N)[q]$. After tensoring with $v$, we obtain an exact triangle of the form 
 \[
 v\otimes u^{\otimes q-1}\xrightarrow{\mathrm{id}_v\otimes a_N} v\otimes u^{\otimes q}\to v\otimes R(G/N)[q] 
 \]
 Moreover, by the projection formula, we can identify $v\otimes R(G/N)$ with the complex $\Ind_N^G(\Res^G_N(v))$. Hence, applying the functor $\mathrm{Hom}_{\K(G;R)}(\mathbb{1},-)$ we obtain an exact sequence 
 \[
 \mathrm{Hom}_{\K(G;R)}(\mathbb{1},v\otimes u^{\otimes q-1})\to \mathrm{Hom}_{\K(G;R)}(\mathbb{1},v\otimes u^{\otimes q})\xrightarrow{\varphi} \mathrm{Hom}_{\K(N;R)}(\mathbb{1},v\restr_N [q]) 
 \]
 where the first map is induced by $a_N$ and the third term is given by the adjunction restriction-induction. Our map $f$ belongs to the middle term, and let us write $f'$ to denote its image under the right hand side map. Recall that $v$ denotes $u_1^{q_1}\otimes \ldots\otimes u_{l-1}^{q_{l-1}}$, hence by Remark \ref{invertibles under restriction} we obtain  $\Res_N(v)=u_{N_1\cap N}^{q_1}\otimes \ldots\otimes u_{N_{l-1}\cap N}^{q_{l-1}}$ where the $N_i\cap N$ are normal subgroups of $N$ of index $p$, hence we can apply our induction step on $l$ to write $f'$ as a polynomial combination on the morphisms $a_{N_i\cap N}$ and $b_{N_i\cap N}$   say 
 \[
 f'=\sum_{i,j}\alpha_{i,j}a_{N_i\cap N}^{\otimes\beta_{i,j}} \otimes b_{N_j\cap N}^{\otimes \delta_{i,j}}
 \]
 with $\alpha_{i,j}$ in $R$ and $\beta_{i,j}$ and $\delta_{i,j}$ non negative numbers. Let $f''\colon \mathbb{1}\to v[q]$ denote the morphism obtained with the same polynomial combination but now with the $a_i$ and $b_i$, that is, 
\[
f''=\sum_{i,j} \alpha_{i,j}a_{N_i}^{\otimes\beta_{i,j}} \otimes b_{N_j}^{\otimes \delta_{i,j}}.
\]
If all the $\delta_{i,j}=0$, then we are done. Otherwise, we may assume that $q$ is even. Then,  
 tensoring together $f''$ and $b_N^{\otimes q/2}\colon \mathbb{1}\to u_N^{\otimes 2q}[-q] $ gives us a map in $\mathrm{Hom}_{\K(G)}(\mathbb{1},v\otimes u^{\otimes q})$. We claim that $f''\otimes b_N^{\otimes 2q} $ is mapped to $f'$ under the map $\varphi$. Note that this will complete the proof since this will tell us that $f''-f$ is the multiplication of $a_N$ with an element $g$ in $\mathrm{Hom}_{\K(G)}(\mathrm{1}, v\otimes u^{q-1})$. In other words, $f=f''b_N^{\otimes q}-g a_N$, but by the induction hypothesis on $q$, we deduce that $g$ is polynomial in the $a_i$ and $b_j$ with $i,j\leq l$. 
It remains to show the claim. Note that our map $f'$ of $RN$--complexes goes to the map $f'\otimes \eta$ of $RG$--complexes via the adjunction restriction-induction, with $\eta$ the unit. On the other hand, the map $b_N^{\otimes}$ is given by $\eta$ in degree 0 and trivial otherwise, with the same $\eta$ as before. We deduce that $f''\otimes b_N^{\otimes q}$ corresponds to $f'\otimes \eta$ under the map $\varphi$, and this completes the claim. 
\end{proof}

\begin{remark}
    In the context of the previous lemma, in the case $(C2)$, we already observed that there are no non-trivial maps  $\mathbb{1}\to u_i[-1]$ for any $i$. In particular, we can trace back the relevant steps in the previous proof to show that there are no non-trivial maps $\mathbb{1}\to u_1^{\otimes q_1}\otimes \ldots\otimes u_l^{\otimes q_l}[s]$ provided that $s$ is odd.   
\end{remark}

  \begin{definition}
  Let $\mathcal{I}$ be a finite set, and $\N^\mathcal{I}=\{\mathcal{I}\to \N\}$. Let $\mathcal{U}=\{u_i\mid i\in \mathcal{I}\}$ be a collection of tensor-invertible objects in $\K(G;R)$. We let $H^{\bullet \bullet}_\mathcal{U}(G;R)$ denote the $(\mathbb{Z}\times \N^\mathcal{I})$--graded ring 
  \[
  \bigoplus_{s\in \mathbb{Z}} \bigoplus_{q\in \N^\mathcal{I}} \mathrm{Hom}_{\K(G;R)}(\mathbb{1},\mathbb{1}(q)[s])
  \]
  where $\mathbb{1}(q)$ is short for $\bigotimes_{i\in \mathcal I} u_i^{\otimes q(i)}$. The multiplication on $H^{\bullet \bullet}_\mathcal{U}(G;R)$ is the one induced by the tensor product in $\K(G;R)$.  We will refer to $\N^\mathcal{I}$ as the monoid of \textit{twists}. 
  \end{definition}

  \begin{remark}
      The ring $H^{\bullet \bullet}_\mathcal{U}(G;R)$ is graded commutative by considering only the parity of the shift and ignoring the twist. For further details, we refer the reader to \cite[Section 12]{BG23b}.
  \end{remark}

We will be mainly interested in the following case.

\begin{definition}\label{def:twisted coh}
We let $\mathcal{N}$ denote the set of all normal subgroups of $G$ of index $p$. Let $\mathbf{U}=\{u_{N,R}\mid N\in \mathcal{N}\}$, where the $u_{N,R}$ are the tensor-invertible objects constructed in Section \ref{section:invertible objects}.  The \textit{twisted cohomology ring of $G$} is the multigraded ring  $H^{\bullet \bullet}_\mathbf{U}(G;R)$ and we usually forget about the subscript and simply write $H^{\bullet \bullet}(G;R)$.
\end{definition}

\begin{corollary}\label{twisted coh is noetherian}
    The twisted cohomology ring $H^{\bullet \bullet}(G;R)$ is a finitely generated $R$--algebra. Therefore, it is Noetherian.  
\end{corollary}

\begin{proof}
    This first part is a consequence of Lemma \ref{aN, bN and CN are generators}. The second claim follows from Hilbert's basis theorem. 
\end{proof}

\begin{remark}   
     Let us highlight that the twisted cohomology ring differs considerably from the group cohomology ring, even for $G=C_p$ in the semisimple case as we will see in the following examples. 
\end{remark}

 Let us discuss twisted cohomology for the cyclic group with $p$ elements over different ground rings.

\begin{example}
   Let $G$ denote the cyclic group $C_p$ and $R$ be a field of characteristic $p$. First, note that $\mathcal{N}$ consists of a single subgroup; the trivial one. Hence $H^{\bullet \bullet}(G;R)$ is a $(\mathbb{Z}\times \mathbb{N})$-graded ring. 
   In this case,  we obtain that $H^{\bullet \bullet}(G;R)$  corresponds to $R[a_N,b_N]$ for $p=2$ with generators in degree $(0,1)$ and $(-1,1)$, respectively. It corresponds to $R[a_N,b_N,c_N]/(c_N^2)$ for $p>2$ generators in degree $(0,1)$ and $(-2,1)$. Here $N$ is the trivial. See \cite[Section 12]{BG23b}. 
\end{example}

\begin{example}
   Let $G$ denote the cyclic group $C_p$ and $R$ be a field of characteristic coprime to $p$. In this case, we obtain that $RG$ is semi-simple. In particular, one verifies that $a_N$ is null-homotopic in this case, but $b_N$ is not. In this case, the ring 
   $H^{\bullet \bullet}(G;R)$  corresponds to $R[b_N]$  and $b_N$ is in degree $(-2,1)$. 
\end{example}

\begin{example}\label{Ex: twisted cohomology}
   Let $G$ denote the cyclic group $C_p$, and let  $R=\mathbb{Z}$. In this case, the reduced part of   $H^{\bullet \bullet}(G;\mathbb Z)$  corresponds to $\mathbb Z[a_N,b_N]/(p\cdot a_N)$, and the generators are in degree $(0,1)$ and $(-2,1)$, respectively. The relations arise precisely when computing the groups $\Hom_{\K(C_p,\mathbb Z)}(\mathbb 1, u_N^{\otimes q})$. 
\end{example}

In the remainder of this section, we will focus on exploring the behavior of the twisted cohomology ring under base change, at least at the level of homogeneous prime spectra.


\begin{remark}
Let $\pp$ be a point in $\mathrm{Spec}(R)$. Recall that  base change along the residue field $R\to k(\pp)$  defines a tt-functor  
\[
\iota^\ast_\pp\colon \K(G;R)\to \K(G;k(\pp)).
\]
Now, consider a normal subgroup $N$ of $G$ of index $p$. As we already observed in Lemma \ref{u_p is invertible}, the tensor invertible object $u_{N,R}$ maps to $u_{N,k(\pp)}$ under the functor $\iota^\ast_\pp$. Let us record the behavior of the maps $a_{N,R}, \, b_{N,R}$ and $c_{N,R}$ under $\iota^\ast_\pp$.
\end{remark}

\begin{lemma}\label{base chage of maps aN, bN, cN}
  Let $N$ be a normal subgroup of $G$ of index $p$. Consider the tt-functor $\iota^\ast_\pp$ obtained by $x\mapsto x\otimes_R k(\pp)$. Then we have the following identifications under this functor:  
  \begin{itemize}
      \item[(i)] $\iota_\pp^\ast(a_{N,R})$ identifies  with $a_{N,k(\pp)}$ in all cases $(C1)$-$(C4)$.
      \item[(ii)] $\iota_\pp^\ast(b_{N,R})$ identifies with $b_{N,k(\pp)}$ in cases $C(1)$, $(C3)$ and $(C4)$, and with $b_{N,k({\pp)}}^{\otimes 2}$ in case $(C2)$. 
      \item[(iii)] $\iota_\pp^\ast(c_{N,R})$ identifies with $c_{N,k(\pp)}$ in case $(C3)$.
  \end{itemize}
\end{lemma}

\begin{proof}
     All cases follow a similar pattern, hence we will only show the identification regarding the map $a_N$, that is, part $(i)$. For readability, let us write $k$ to denote $k(\pp)$. Since $\iota_\pp^\ast$ is a tt-functor, we get that $\mathbb{1}_{\K(G;R)}=R$ goes to $\mathbb{1}_{\K(G,k)}=k$ via the canonical map.  Moreover, we showed in Lemma \ref{u_p is invertible} that $\iota_\pp^\ast(u_{N,R})$ agrees with $u_{N,k}$ via the isomorphism $R(G/N)\otimes_Rk\simeq k(G/N)$ point-wise. In particular, we get a diagram 
     \begin{equation*}
         \begin{tikzcd}
         R \otimes_R k \arrow[d,"a_{N,R}\otimes_R  \mathrm{id}_k"'] \arrow[r,"\simeq"] & k \arrow[d,"a_{N,k}"]\\
             u_{R,N}\otimes_Rk \arrow[r,"\simeq"]  & u_{N,k}          \end{tikzcd}
     \end{equation*}
     and it is easy to verify that this square is commutative. 
\end{proof}

Since the functor $\iota_\pp^\ast\colon \K(G;R)\to \K(G;k(\pp))$ is symmetric monoidal, we obtain a morphism of $(\mathbb{Z}\times \mathbb{N}^\mathcal{N})$-graded rings 
\[
\Theta_\pp\colon H^{\bullet \bullet}(G;R)\otimes_R k(\pp)\to H^{\bullet \bullet}(G;k(\pp)).
\]
Let us record the following property of the morphism $\Theta_\pp$.

\begin{corollary}\label{coro: injectivity by base change}
    The map $\Theta_\pp$ is power surjective, that is, for any $\alpha$ in $H^{\bullet \bullet}(G;k(\pp))$, there is natural number $n$, such that $\alpha^{\otimes n}$ is in the image of $\Theta_\pp$. In particular, we get that $\Theta_\pp$ induces an injective map 
    \[
    \mathrm{Spec}^h(H^{\bullet \bullet}(G;k(\pp)))\to \mathrm{Spec}^h(H^{\bullet \bullet}(G;R)\otimes_R k(\pp)). 
    \]
\end{corollary}

\begin{proof}
    This is an immediate consequence of Lemma \ref{base chage of maps aN, bN, cN}, since the maps $a_N$, $b_N$, and $c_N$ are algebra generators. In particular, in case $(C3)$, we get that for $b_{N,k(\pp)}$, the required $n$ is precisely $2$. It is probably worth highlighting the following scenario: in case $(C4)$, we might have a morphism $c_{N,k(\pp)}$ (depending on the characteristic of $k(\pp)$), but we have no $c_{N,R}$. This is not a problem since such $c_{N,k(\pp)}$ are nilpotent (see Remark \ref{cN is nilpotent}).
\end{proof}

\section{Comparing Balmer and homogeneous spectra}\label{sec: comparing}

Recall that $G$ denotes a $p$-group and $R$ denotes a commutative Noetherian ring subject to Convention \ref{convention on R}. In this section, we relate the Balmer spectrum of $\K(G;R)$ and the homogeneous spectrum of $H^{\bullet\bullet}(G;R)$ via the comparison map. Let us recall the construction of this map and refer to \cite[Theorem 5.3]{Bal10} and \cite[Proposition 13.4]{BG23b} for a proof.

\begin{lemma}\label{lemma:comparison map}
    There is a continuous comparison map
    \begin{align*}
     \mathrm{Comp}_{G,R}\colon \mathrm{Spc}(\K(G;R)) & \to \mathrm{Spec}^h(H^{\bullet\bullet}(G;R)) \\ 
        \mathcal{P} & \mapsto \langle f\colon \mathbb{1}\to \mathbb{1}(q)[s]\mid \mathrm{cone}(f)\not\in \mathcal{P}\rangle. 
    \end{align*} 
    where the right hand side denotes the ideal generated by the maps $f$. 
\end{lemma}

Let us record that this comparison map is compatible with base change along the residue fields of $R$.  In fact, the comparison map is natural with respect to tt-functors and the same proof as \cite[Theorem 5.3]{Bal10} applies in this case. 

\begin{proposition}\label{base change and comparison map}
    Let $\pp$ be a prime of $R$. Then base change along the residue field $R\to k(\pp)$ makes the following diagram commutative. 
\begin{center}
    \begin{tikzcd}[column sep = 5em]
       \mathrm{Spc}(\K(G;k(\pp))) \arrow[r,"\mathrm{Spc}(\iota_\pp^\ast)"] \arrow[d,"\mathrm{Comp}_{G,k(\pp)}"'] & \mathrm{Spc}(\K(G;R)) \arrow[d,"\mathrm{Comp}_{G,R}"] \\ 
 \mathrm{Spec}^h(H^{\bullet\bullet}(G;k(\pp))) \arrow[r,"\mathrm{Spec}(\iota_\pp^\ast)"] & \mathrm{Spec}^h(H^{\bullet\bullet }(G;R))
    \end{tikzcd}
\end{center}
\end{proposition}

\begin{lemma}\label{comaprison is injective}
    Let $E$ be an elementary abelian $p$-group. Then the comparison map 
    \[
    \mathrm{Comp}_E\colon \mathrm{Spc}(\K(E;R))  \to \mathrm{Spec}^h(H^{\bullet\bullet}(E;R))
    \]
    is injective. 
\end{lemma}

\begin{proof}
    By Proposition \ref{base change and comparison map}, we obtain  the following commutative square
    
\begin{center}
    \begin{tikzcd}
       \displaystyle \bigsqcup_{\pp\in\mathrm{Spec}(R)} \mathrm{Spc}(\K(E;k(\pp))) \arrow[r] \arrow[d] & \mathrm{Spc}(\K(E;R)) \arrow[d] \\ 
\displaystyle \bigsqcup_{\pp\in\mathrm{Spec}(R)} \mathrm{Spec}^h(H^{\bullet\bullet}(E;k(\pp))) \arrow[r] & \mathrm{Spec}^h(H^{\bullet\bullet}(E;R)).
    \end{tikzcd}
\end{center}
 We claim that the left-vertical map is injective. Indeed, if the characteristic of $k(\pp)$ is $p$, then the corresponding component is injective by  \cite[Proposition 15.1]{BG23b}; if the characteristic of $k(\pp)$ is not $p$, then $\mathrm{Spc}(\K(E;k(\pp)))$ is simply a point, and the respective component map is trivially injective. Now, note that the bottom map is injective since the components are inclusions into the fibers induced by the map $\mathrm{Spec}^h(H^{\bullet\bullet}(E;R))\to \mathrm{Spec}^h(R)$ (see Corollary \ref{coro: injectivity by base change}).  It follows that the top map is injective as well. Moreover, by Proposition \ref{Spc as a set}, the top map is a bijection. We deduce that the comparison map is injective by the commutativity of the square.    
\end{proof}

\begin{remark}
   One should not expect the previous result to hold for more general $p$-groups, not even when $R$ is a field, as discussed in \cite{BG23b}. 
\end{remark}


\section{A Dirac scheme structure}\label{sec:dirac}


Let $E$ be an elementary abelian $p$-group. In this section, we define a special cover of the Balmer spectrum $\mathrm{Spc}(\K(E;R))$, along with a sheaf of graded rings, which together exhibit $\mathrm{Spc}(\K(E;R))$ as a Dirac scheme. This generalizes the result of Balmer and Gallauer in the case where the ground ring is a field.  We need some preparations.

\begin{definition}\label{def:open(f)}
  Let $G$ be a $p$-group. Let $f$ be a homogeneous map in the twisted cohomology ring $H^{\bullet\bullet}(G;R)$. Following \cite[Notation 13.7]{BG23b}, we let $\mathrm{open}(f)\coloneqq\mathrm{open}(\mathrm{cone}(f))$ denote the open subset of $\mathrm{Spc}(\K(G;R))$ defined as the preimage under $\mathrm{Comp}_G$ of the principal open $Z(f)^c=\{\pp\mid f\not\in \pp\}$. That is, $\mathrm{open}(f)$ is given by $\{\mathcal{P}\mid \mathrm{cone}(f)\in \mathcal{P}\}$. In other words, the subset $\mathrm{open}(f)$ is the open locus of $\mathrm{Spc}(\K(G;R))$ where $f$ is invertible. 
\end{definition}

\begin{proposition}\label{cone of an kills cone of bn}
    Let $N$ be a subgroup of $G$ of index $p$. Then 
    \[
    \mathrm{cone}(a_{N,R})\otimes \mathrm{cone}(b_{N,R})
    \]
    is trivial. In particular, 
    \[
    \mathrm{open}(a_{N,R} ) \cup \mathrm{open}( b_{N,R})=\mathrm{Spc}(\K(G;R)).
    \]
\end{proposition}

\begin{proof}
     For a prime $\pp$ of $R$, consider the residue field $R\to k(\pp)$. Note that 
     \[
     \mathrm{cone}(a_{N,R}\otimes_R k(\pp))\cong \mathrm{cone}(a_{N,k(\pp)}) \, \textrm{ and } \, \mathrm{cone}(b_{N,R}\otimes_R k(\pp))\cong \mathrm{cone}(b_{N,k(\pp)}^{\otimes2'}) 
     \]
     where 2' is 2 in case (C2) and 1 otherwise.
     Moreover, note that $\mathrm{cone}(a_{N,k(\pp)})\otimes \mathrm{cone}(b_{N,k(\pp)}^{\otimes 2'})=0$. Indeed, if the characteristic of $k(\pp)$ is $p$, this follows by \cite[Corollary 13.3]{BG23b}; if the characteristic of $k(\pp)$ is coprime to $p$, this follows by the fact that $b_N$ is a quasi-isomorphism which is enough since it lies in the semi-simple case. Hence, the first claim follows by jointly-conservativity of the family of functors $\{\iota_\pp^\ast\}_{\pp\in \Spec(R)}$ \cite[Theorem 4.8]{Gom25}. For the second claim, note that $\mathrm{open}(a_{N,R} ) \cup \mathrm{open}(b_{N,R})=\mathrm{open}(\mathrm{cone}(a_{N,R})\otimes \mathrm{cone}(b_{N,R}))=\mathrm{open}(0)=\mathrm{Spc}(\K(G;R)).$
\end{proof}

For the rest of this section, we focus on an elementary abelian $p$-group $E$. Recall that $\mathcal{N}(E)$ denotes the set of all subgroups of $E$ of index $p$.

\begin{definition}\label{def:u(H)}
Let $H$ be a subgroup of $E$. Then we define the open subset $U_E(H)$ of $\mathrm{Spc}(\K(E;R))$ given by  
\[
U_E(H)\coloneqq \left( \displaystyle \bigcap_{N\in\mathcal{N}, \, H\not\leq N} \mathrm{open}(a_{N,R}) \right) \cap \left(\bigcap_{N\in\mathcal{N}, \, H\leq N} \mathrm{open}(b_{N,R})\right).
\]
\end{definition}

\begin{definition}\label{def:LH}
    For $H\leq E$, let $S_H\subset H^{\bullet\bullet}(E;R)$ be the multiplicative subset generated by the set 
    \[
    \{ a_{N,R}\mid N\in\mathcal{N}(E),\, H\not \leq N\}\cup \{ b_{N,R}\mid N\in\mathcal{N}(E),\, H \leq N\}.
    \]
    The tt-category $\mathcal{L}(H;R)$ is defined as the Verdier quotient $\K(E;R)/I$ where $I$ denotes the tt-ideal generated by the collection $\{\mathrm{cone}(f)\mid f\in S_H\}$.  Let $\mathcal{O}^\bullet_E(H)$ denote the $\mathbb{Z}$--graded ring $ \mathrm{End}^\bullet_{\mathcal{L}(H;R)}(\mathbb{1})$. 
\end{definition}

\begin{remark}\label{rmk:spc of L(H)}
 Imitating  Construction 14.12 in \cite{BG23b}, we consider the open subset $U=U_E(H)$ which agrees with $\bigcap_{f\in S_H} \mathrm{open}(f)$. Just as in the field case, one verifies that the complement of $U$ agrees with $Z\coloneqq\bigcup_{f\in S_H} \mathrm{supp}(\mathrm{cone}(f))$. In particular, the tt-ideal $\K(E;R)_Z$ supported on $Z$ is the one generated by $\{\mathrm{cone}(f)\mid f\in S_H\}$. Therefore the tt-category $\K(E;R)_U=(\K(E;R)/\K(E;R)_Z)^\natural$ is the idempotent completion of $\mathcal{L}(H;R)$. It follows that $\mathrm{Spc}(\mathcal{L}(H;R))\cong \mathrm{Spc}(\K(G;R)_U)\cong U$, see \cite{BF11} for a proof of the last homeomorphism. 
\end{remark}

\begin{remark}\label{rem:morphism as fractions in LH}
    Note that the $\mathbb{Z}$--graded ring $\mathcal{O}^\bullet_E(H)$ can be identified with the zero-twist part of the localization of $H^{\bullet\bullet}(G;R)$ on $S_H$, in symbols: 
\[
\mathcal{O}^\bullet_E(H)=(H^{\bullet\bullet}(G;R)[S_H^{-1}])_{0\textrm{-twist}}
\]
 see \cite[Definition 14.9]{BG23b}. In particular, we can consider the elements of $\mathcal{O}^\bullet_E(H)$ as fractions $\frac{f}{g}$ with $f\colon \mathbb{1}\to \mathbb{1}(q)[s]$ and $g\colon \mathbb{1}\to \mathbb{1}(q)[s'] \in S_H$, for some $q\in \mathcal{N}(E)$ and $s,s'\in\mathbb{Z}$. On the other hand, the morphisms in $\mathcal{L}(E;R)$ can be described as fractions $\frac{f}{g}$ where $g\colon \mathbb{1}\to \mathbb{1}(q)[s]$ is a map in $S_H$ and $f\colon x\to y\otimes \mathbb{1}(q)[1]$ is any morphism in $\K(E;R)$, see \cite[Construction 14.12]{BG23b} and the references therein. 
\end{remark}

\begin{remark}\label{rem:comparison of LH and KE}
Let $H$ be a subgroup of $E$.  Consider the comparison map 
\[
\rho_{\mathcal{L}(H;R)}\colon \mathrm{Spc}(\mathcal{L}(H;R))\to \mathrm{Spec}^h(\mathcal{O}^\bullet_E(H))
\]
constructed in  \cite{Bal10} corresponding to the tensor invertible complex $\mathbb{1}[1]$. We obtain  the following 
commutative square.  

\begin{center}
   \begin{tikzcd}
    \mathrm{Spc}(\mathcal{L}(H;R)) \arrow[d,"\rho_{\mathcal{L}(H;R)}"'] \arrow[r] & \mathrm{Spc}(\K(E;R)) \arrow[d,"\mathrm{Comp}_E"] \\
    \Spec^h(\mathcal{O}^\bullet_E(H)) \arrow[r] &  \Spec^h(H^{\bullet\bullet}(E;R)) 
\end{tikzcd} 
\end{center}
 \noindent where the horizontal maps are induced by the respective localizations. This square is commutative by the definition of the comparison maps. 
\end{remark}

\begin{theorem}\label{thm:comp LH is an homeomorphism}
    Let $H$ be a subgroup of $E$. Then the comparison map 
    \[
    \rho_{\mathcal{L}(H;R)}\colon \mathrm{Spc}(\mathcal{L}(H;R))\to \Spec^h(\mathcal{O}^\bullet_E(H))
    \]
    of Remark \ref{rem:comparison of LH and KE} is a homeomorphism. In particular, $\Spec^h(\mathcal{O}^\bullet_E(H))\cong U_E(H)$.
\end{theorem}

\begin{proof}
   Using the commutative square from Remark \ref{rem:comparison of LH and KE} and the injectivity of the comparison map $\mathrm{Comp}_G$ (see Lemma \ref{comaprison is injective}) we deduce that the comparison map  $\rho_{\mathcal{L}(H;R)}$ is injective as well. Since $\mathrm{End}^\ast_{\mathcal{L}(H;R)}(\mathbb{1})= \mathcal{O}^\bullet_E(H)$ is Noetherian,  $\rho_{\mathcal{L}(H;R)}$ is surjective; see \cite{Bal10}. Thus \cite[Theorem 10.8]{sanders2025tensor} gives us that $\rho_{\mathcal{L}(H;R)}$ is a homeomorphism. 
\end{proof}

\begin{remark}\label{rem:LE under base change}
 Let $\pp$ be a prime in $R$, and consider the residue field $R\to k(\pp)$. Fix a subgroup $H$ of $E$. Let $g$ be a map in the multiplicative subset $S_{H,R}$ of $H^{\bullet\bullet}(E;R)$ from Definition \ref{def:LH}. Note that the map $k(\pp)\otimes g$, obtained by base change along $R\to k(\pp)$, lies in the multiplicative set $S_{E,k(\pp)}$. We deduce that the functor $\K(E;R)\to \K(E;k(\pp))$ fits into a commutative diagram
\[\begin{tikzcd}
    \K(E;R)\arrow[r] \arrow[d] & \K(E;k(\pp))\arrow[d]\\
    \mathcal{L}(H;R) \arrow[r] & \mathcal{L}(H;k(\pp))
\end{tikzcd}\]
 \noindent where the bottom map is also induced by base change along $R\to k(\pp)$.
\end{remark}

\begin{lemma}\label{lemma:open cover}
    The collection $\{U_E(H)\mid H\leq E\}$ is an open cover of the Balmer spectrum of $\K(E;R)$.
\end{lemma}

\begin{proof}
    For each prime $\pp$ in $R$ consider the commutative diagram induced by base change along the residue field $R\to k(\pp)$ given in Remark \ref{rem:LE under base change}. We obtain a commutative diagram of topological spaces   
\begin{center}
    \begin{tikzcd}
       \displaystyle \bigsqcup_{\pp\in\mathrm{Spec}(R)} 
 \bigsqcup_{H\leq E} \mathrm{Spc}(\mathcal{L}(H;k(\pp))) \arrow[r] \arrow[d] & \displaystyle \bigsqcup_{\pp\in\mathrm{Spec}(R)} \mathrm{Spc}(\K(E;k(\pp))) \arrow[d] \\ 
\displaystyle  \bigsqcup_{H\leq E} \mathrm{Spc}(\mathcal{L}(H;R)) \arrow[r] & \mathrm{Spc}(\K(E;R))
    \end{tikzcd}
\end{center}
where the right-hand vertical map is bijective by Proposition \ref{Spc as a set}, and the top map is surjective by \cite[Proposition 13.11]{BG23b}. We deduce that the bottom is surjective. Moreover, we already know that $\mathrm{Spc}(\mathcal{L}(H;R))\cong U_E(H)$ under the map induced by localization, hence the result follows. 
\end{proof}

Now, let $\mathcal{O}^\bullet_E$ be the sheaf  of graded rings on $\mathrm{Spc}(\K(E;R))$ obtained by sheafifying the presheaf 
\[
U\mapsto \mathrm{End}^\ast_{\K(E;R)(U)}(\mathbb{1})
\]
where $\K(E;R)(U)$ is the localization over a quasi-compact open $U\subset \mathrm{Spc}(\K(E;R))$, i.e.,  $\K(E;R)(U)=(\K(E;R)/\K(E;R)_Z)^\natural$, with $Z=\mathrm{Spc}(\K(E;R))\backslash U$. By Theorem \ref{thm:comp LH is an homeomorphism} and Lemma \ref{lemma:open cover}, we have exhibited an open cover of $\mathrm{Spc}(\K(E;R))$ which gives us the following result. 

\begin{corollary}\label{cor: spc is a dirac scheme}
    Let $E$ be an elementary abelian $p$-group. Then $(\mathrm{Spc}(\K(E;R),\mathcal{O}_E^\bullet)$ is a Dirac scheme in the language of \cite{HP23}.
\end{corollary}

\begin{corollary}\label{cor: comp is open immersion}
    The comparison map 
    \[
    \mathrm{Comp}_E\colon \mathrm{Spc}(\K(E;R))\to \Spec^h(H^{\bullet\bullet}(E;R))
    \]
    is an open immersion.
\end{corollary}

\begin{proof}
    Note that the comparison map is an open immersion locally. Indeed,  for the open cover $\{U_E(H)\mid H\leq E\}$  of the spectrum, the restriction of the comparison map is a homeomorphism to the  open subset of $\Spec^h(H^{\bullet\bullet}(E;R))$ given by 
\[
U(H)'=\bigcap_{H\not \leq N}Z(a_N)^c \cap \bigcap_{H\leq N}Z(b_N)^c
\] 
 where $Z(f)^c$ denotes the principal open defined by $f$ in $\Spec^h(H^{\bullet\bullet}(E;R))$; see Theorem \ref{thm:comp LH is an homeomorphism}. Since the comparison map is injective by Lemma \ref{comaprison is injective} and a locally open embedding, we obtain that it is an open embedding.
\end{proof}


\section{Examples}\label{sec: examples}

In this section, we use the tools developed so far to describe the Balmer spectrum of $\K(G,\mathbb Z)$  for $G$ a $p$-group. We also give explicit pictures for cyclic $p$-groups. 

\begin{recollection}
    Recall that the \textit{modular locus of $\Spc(\K(G;R))$} is the preimage of $(p)$ under the comparison map $\Spc(\K(G;R))\to \Spec(\mathbb Z)$. The \textit{ordinary locus of  $\Spc(\K(G;R))$} is the set theoretic complement of the modular fiber; see Definition \ref{def-modular-locus}.  
\end{recollection}

\begin{example}\label{ex-main}
  Let $G$ be a $p$-group, and set $n=|G|$. We now describe $\Spc(\K(G;\mathbb{Z}))$. By Theorem \ref{thm-modular-locus}, the modular locus $\Spc(\K(G;\mathbb{Z}))_p$ identifies with $\Spc(\K(G;\mathbb{F}_p))$ which has been computed in \cite{BG23b}. On the other hand, by Lemma \ref{lemma-ordinary-locus}, its ordinary locus corresponds to $\Spec(\mathbb{Z}[1/n])$. Then it's sufficient to explain how these two pieces are glued. Given that the modular locus is closed, we only have to describe the possible specializations from points in the ordinary locus. By Theorem \ref{thm-closure-R[1/n]}, we obtain that that such specializations must lie completely in the cohomological open. The latter identifies with $\Spec^h(H^\bullet(G;\mathbb Z))$ by Corollary \ref{coro-coh-open}. This gives a description of the Balmer spectrum of $\K(G;\mathbb Z)$ in terms of $\Spc(\K(G;\mathbb{F}_p))$, $\Spec^h(H^\bullet(G;\mathbb Z))$ and $\Spec(\mathbb Z[1/n])$. The following figure shows how the gluing of the modular locus and the ordinary locus occur. 
  
    \begin{figure}[h]
     \centering
    \begin{tikzpicture}
  \def\regionA{
    [rounded corners=1cm] (-3.6,-2.2) rectangle (2,1.3)
  }

  \def\regionB{
    (-2,0)
    .. controls (-2,1.35) and (-0.8,1.35) .. (0.2,1.1)
    .. controls (1.2,0.84) and (1.55,0.8) .. (2,0.8)
    -- (3.45,0.8)
    .. controls (4.45,0.8) and (4.45,-0.8) .. (3.45,-0.8)
    -- (2,-0.8)
    .. controls (1.55,-0.8) and (1.2,-0.84) .. (0.2,-1.1)
    .. controls (-0.8,-1.35) and (-2,-1.35) .. cycle
  }

  \def\regionC{
    (-1.2,0)
    .. controls (-1.2,1.05) and (-0.4,1.05) .. (0.4,0.85)
    .. controls (1.2,0.68) and (1.6,0.7) .. (2,0.7)
    -- (3.4,0.7)
    .. controls (4.35,0.7) and (4.35,-0.7) .. (3.4,-0.7)
    -- (2,-0.7)
    .. controls (1.6,-0.7) and (1.2,-0.68) .. (0.4,-0.85)
    .. controls (-0.4,-1.05) and (-1.2,-1.05) .. cycle
  }

  \fill[OliveGreen!15] \regionA;
  \fill[purple!15] \regionB;
  \fill[magenta!10] \regionC;
  \fill[brown!35] (3.05,0) ellipse (0.99cm and 0.6cm);

  \draw[OliveGreen!50!black, thick] \regionA;
  \draw[purple!50!black, thick] \regionB;
  \draw[magenta!50!black, thick] \regionC;
  \draw[brown!60!black, thick]
    (3.05,0) ellipse (0.99cm and 0.6cm);

  \node at (-2.8,1) {$A$};
  \node at (-0.8,1.05) {$B$};
  \node at (-0.4,0) {$C$};
  \node at (3.05,0) {$D$};
\end{tikzpicture}  
     \caption{The Balmer spectrum of $\K(G;\mathbb Z)$: Region $A$ ({\color{OliveGreen}OliveGreen})
     corresponds to the modular locus; Region $B$ ({\color{purple}purple}) is the cohomological open; Region $C$ ({\color{magenta}magenta}) represents the closure of the ordinary locus; Region $D$ ({\color{brown}brown}) is the ordinary locus.} 
     \label{fig-integralspectrum}
 \end{figure}

We warn the reader that the previous picture is only intended to show the gluing of the pieces, and not about differences on the height of the different regions. 
\end{example}

Let us focus now on $\K(G,\mathbb Z)$ for $G$ a cyclic $p$-group.  In this case, we give a more explicit picture of the spectrum. 

\begin{recollection}\label{spc for cp over fields}
Let $k$ denote a field of positive characteristic $p$, and let $G=C_{p^n}$ be the cyclic group of order $p^n$. By \cite[Section 8]{BG23b} we  know that the spectrum of $\K(C_{p^n},k)$ has the form 
\begin{equation}\label{eq:Spc Cn}
\Spc(\K(C_{p^n};k))=\qquad
\vcenter{\xymatrix@R=1em@C=.7em{
{\scriptstyle\mm_0\kern-1em}
& {\bullet} \ar@{-}@[Gray] '[rd] '[rr] '[drrr]
&&
{\bullet}
&{\kern-1em\scriptstyle\mm_1}
&&{\scriptstyle\mm_{n-1}\kern-1em}&\bullet&&\bullet&{\kern-1em\scriptstyle\mm_n}
\\
&{\scriptstyle\pp_1\kern-1em}& {\bullet}&&& {\cdots}& \ar@{-}@[Gray] '[ru] '[rr] '[rrru] &&\bullet&{\kern-1em}\scriptstyle\pp_{n}}}
\end{equation} 
consisting of $2n+1$ points and specialization relations pointing upward. In particular,  the closed subsets are   those that contain a $\pp_i$ only if they contain $\mm_{i-1}$ and $\mm_i$.
Hence the $\mm_i$ are closed points and the $\pp_i$ are generic points of the $n$ irreducible V-shaped closed subsets $\{\mm_{i-1},\pp_i,\mm_i\}$. Moreover, the points~$\pp_i$ and~$\mm_i$ can be described as follows. Consider a maximal normal series  of $G$
\[
0=N_n<N_{n-1}<\cdots<N_0=G.
\]
Then the points $\pp_i$ and $\mm_i$ in~$\Spc(\K(G,k))$ are given by 
\[
\mm_i=(\check{\Psi}^{N_i})^{-1}(0)
\mbox{ and }
\pp_i=(\check{\Psi}^{N_i})^{-1}(\mathbf{D}_{\mathrm{perf}}(\K(G/N_i,k)))
\]
where $\check{\Psi}^{N}=\Upsilon_{G/N}\circ\Psi^N\colon \K(G,k)\to\K(G/N,k)\to\mathbf{D}_b(\K(kG/N))$ is the functor from \cite[Notation 6.10]{BG22b}.
\end{recollection}

\begin{example}
Let $G=C_p$ be the cyclic group with $p$ elements. Then the Balmer spectrum of $\K(G,\mathbb Z)$ is homeomorphic to the space 
\begin{equation}
\kern2em\vcenter{\xymatrix@C=.0em@R=.4em{
{\color{Brown}\overset{(2)}{\bullet}}   \ar@{-}@[Brown][rrrrrrrrrrrdddd]  &&& {\color{Brown}\overset{(3)}{\bullet}}  \ar@{-}@<.1em>@[Brown][rrrrrrrrdddd]
&& \ldots
&& {\color{OliveGreen}\overset{\mm_0}{\bullet}} \ar@{-}@[OliveGreen][rdd] 
&& {\color{OliveGreen}\overset{\mm_1}{\bullet}} \ar@{-}@[OliveGreen][ldd] \ar@{~}@[Gray][rrdddd]
&&  \cdots  
&&  {\color{Brown}\overset{(q)}{\bullet}}  \ar@{-}@[Brown][lldddd] \ar@{.}@[Gray][r] 
&& \cdots
\\ \\
&&
&&
&&
&& {\color{OliveGreen}\underset{\pp_1}{\bullet}} 
&
& 
& 
&& 
&& 
&&
&&
&&
\\ \\
&&&&&&&&& && {\color{Brown}\bullet_{(0)}}\kern-1em
&&&&&&&& 
}}\kern-1.11em
\end{equation}
where the green part corresponds to the modular locus and the brown part to the ordinary locus. Indeed,  it is enough to verify that $(0)$  specializes to the closed point $\mm_1$ in the modular locus. In this case, the cohomological open corresponds to $\Spec^h(\mathbb Z[x]/(px))$, which identifies with the space in  the picture above obtained by removing $\mm_0$.

In this case, we can also use twisted cohomology  to describe the topology of $\K(C_p,\mathbb Z)$. Indeed, by Example \ref{Ex: twisted cohomology} we have that the reduced part of the twisted cohomology $H^{\bullet\bullet}(C_p;\mathbb Z)$ is given by $\mathbb{Z}[a_N,b_N]/(p\cdot a_N)$ where $N$ denotes the trivial subgroup of $C_p$. In this case, the cover of $\Spc(C_p;\mathbb Z)$ given in Lemma \ref{lemma:open cover} has two opens, namely $U(N)$ which is the open locus where $b_N$ is invertible and $U(C_p)$ that corresponds to the open locus where $a_N$ is invertible. One verifies that $U(N)$ is all the brown region together with $\mm_1$; $U(C_p)$ contains the points $\mm_0$ and $\pp_0$. 

Moreover, one can verify that the reduced part of $\mathcal{O}^\bullet(N) $ corresponds to $ \mathbb Z[x]/(p\cdot x)$ and $\mathcal{O}^\bullet(C_p)$ and the reduced part to $ \mathbb F_p[y]$. This recovers the description given above.
\end{example}

\begin{example}
Let $G=C_{p^n}$ be the cyclic group of order $p^n$. Then the Balmer spectrum of $\K(G,\mathbb Z)$ is homeomorphic to the space 
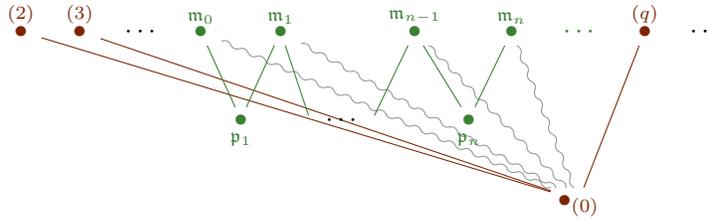
\begin{equation*}
\kern2em\vcenter{\xymatrix@C=.0em@R=.4em{
{\color{Brown}\overset{(2)}{\bullet}}    \ar@{-}@[Brown][rrrrrrrrrrrrrrrrdddd]
&& {\color{Brown}\overset{(3)}{\bullet}}  \ar@{-}@[Brown][rrrrrrrrrrrrrrdddd]
&& \cdots
&& {\color{OliveGreen}\overset{\mm_0}{\bullet}} \ar@{-}@[OliveGreen][rdd] 
&& {\color{OliveGreen}\overset{\mm_1}{\bullet}} \ar@{-}@[OliveGreen][ldd] \ar@{-}@[OliveGreen][rdd] 
&& 
&& {\color{OliveGreen}\overset{\mm_{n-1}}{\bullet}} \ar@{-}@[OliveGreen][ldd] \ar@{-}@[OliveGreen][rdd] 
&& {\color{OliveGreen}\overset{\mm_n}{\bullet}} \ar@{-}@[OliveGreen][ldd] \ar@{~}@[Gray][rrdddd] 
&& {\color{OliveGreen}\cdots}
&& {\color{Brown}\overset{(q)}{\bullet}}  \ar@{-}@[Brown][lldddd]
&& \cdots 
\\ \\ 
&&  
&& 
&&
&  {\color{OliveGreen}\underset{\pp_1}{\bullet}} & 
&& \cdots
&& 
& {\color{OliveGreen}\underset{\pp_n}{\bullet}}&  
&&  
&&  
\\ \\ 
&&
&& 
&& 
&& 
&& 
&& 
&& 
&& {\color{Brown}{\bullet_{(0)}}}   
&& 
}}\kern-1.11em
\end{equation*}
where the green part corresponds to the modular locus and the brown part to the ordinary locus.  Indeed, the modular fiber is described in Recollection \ref{spc for cp over fields}, and the ordinary locus is as before. We only need to verify that $(0)$ specializes to $\mm_n$, but this follows in a similar way as before since the space obtained from the above picture by removing the points $\mm_0,\ldots, \mm_{n-1},\pp_1,\ldots, \pp_m$ is the spectrum of $H^\bullet(C_{p^n};\mathbb Z)$.
\end{example}

By Theorem \ref{thm-red-p-groups}, we can also deal with general finite groups. As an example, we describe the spectrum of $\K(G,\mathbb Z)$ for  $G=C_{p_1}\times C_{p_2}$ for primes $p_1\not= p_2$. 

\begin{example}
    $G=C_{p_1}\times C_{p_2}$ for primes $p_1\not= p_2$.  Then the Balmer spectrum of $\K(G,\mathbb Z)$ is homeomorphic to the space 
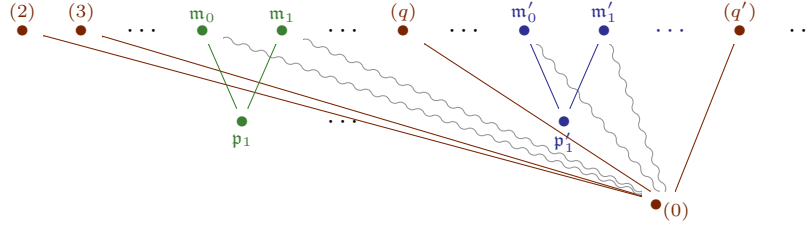
\begin{equation*}%
\kern2em\vcenter{\xymatrix@C=.0em@R=.4em{
{\color{Brown}{\bullet}}    \ar@{-}@[Brown][rrrrrrrrrrrrrrrrrrrrdddd]
&& {\color{Brown}{\bullet}}  \ar@{-}@[Brown][rrrrrrrrrrrrrrrrrrdddd]
&& \cdots
&& {\color{OliveGreen}\overset{\mm_0}{\bullet}} \ar@{-}@[OliveGreen][rdd] 
&& {\color{OliveGreen}\overset{\mm_1}{\bullet}} \ar@{-}@[OliveGreen][ldd]  \ar@{~}@[Gray][rrrrrrrrrrrrdddd]
&& \cdots
&& {\color{Brown}{\bullet}} \ar@{-}@[Brown][rrrrrrrrdddd]
&& \cdots
&& {\color{Blue}\overset{\mm_{0}'}{\bullet}}  \ar@{-}@[Blue][rdd]  
&& {\color{Blue}\overset{\mm_1'}{\bullet}} \ar@{-}@[Blue][ldd] \ar@{~}@[Gray][rrdddd] 
&& {\color{Blue}\cdots}
&& {\color{Brown}{\bullet}}  \ar@{-}@[Brown][lldddd]
&& \cdots 
\\ \\ 
&&  
&& 
&&
&  {\color{OliveGreen}\underset{\pp_1}{\bullet}} & 
&& \cdots
&& 
&&
&&
& {\color{Blue}\underset{\pp_1'}{\bullet}}&  
&&  
&&  
\\ \\ 
&&
&& 
&& 
&& 
&&
&&
&& 
&& 
&& 
&& {\color{Brown}{\bullet_{(0)}}}   
&& 
}}\kern-1.11em
\end{equation*}
where the green part corresponds to the fiber over $(p_1)$, the blue part corresponds to the fiber over $(p_2)$ and the brown part to the ordinary locus. In this case, the family of prime-power order subgroups of $G$ is 
 $\{C_{p_1}, C_{p_2},0\}$. By Theorem \ref{thm-red-p-groups}, we have that $\Spc(\K(G;\mathbb Z))$ is a quotient of 
 \[
 \Spc(\K(C_{p_1};\mathbb Z))\sqcup \Spc(\K(C_{p_2};\mathbb Z))\sqcup \Spec(\mathbb{Z})
 \]
and the identifications are made via the restriction to the trivial subgroup. From here, one verifies that the three copies of  $\Spec(\mathbb{Z}[1/p_1\cdot p_2])$ in the coproduct above identify into a single copy which corresponds to the ordinary locus of $\Spc(\K(G;\mathbb Z))$. This also gives us that $(p_1)$ identifies with $\mm_1$, and $(p_2)$ identifies with $\mm_1'$. This completes the description of the previous picture. 
\end{example}


\bibliographystyle{alpha}
\bibliography{mybibfile} 

\end{document}